\documentclass[preprint]{elsarticle}

\usepackage{lineno,hyperref}
\modulolinenumbers[5]

\journal{Neural Networks}

\usepackage{amsthm,amssymb}

\usepackage{todonotes,marginnote}

\usepackage{tikz}
\usetikzlibrary{shapes,arrows,positioning}

\usepackage[utf8]{inputenc}
\usepackage{amsmath,amssymb,amsthm,commath,esint,tikz-cd,tikz, mathtools, physics}
\usepackage[mathscr]{euscript}
\usepackage{amsfonts}
\usepackage{subcaption}
\usepackage{hyperref}
\usepackage{graphicx}
\usepackage{float}
\usepackage{fancyhdr}
\usepackage{bm}

\usepackage{todonotes}

\newcommand{\ck}[1]{\norm{#1}_{W^{k,\infty}}}
\newcommand{\cki}[1]{\norm{#1}_{W^{k,\infty}(I_i^N)}}
\newcommand{\linfi}[1]{\norm{#1}_{L^{\infty}(I_i^N)}}
\newcommand{\ckunit}[1]{\norm{#1}_{W^{k,\infty}([0,1]^d)}}
\newcommand{\linfunit}[1]{\norm{#1}_{L^{\infty}([0,1]^d)}}

\newcommand{\N}{{\mathbb{N}}}

\newtheorem{theorem}{Theorem}[section]

\newtheorem{remark}[theorem]{Remark}
\newtheorem{definition}[theorem]{Definition}
\newtheorem{lemma}[theorem]{Lemma}
\newtheorem{proposition}[theorem]{Proposition}
\newtheorem{corollary}[theorem]{Corollary}

\begin{document}

\begin{frontmatter}

\title{On the approximation of functions by tanh neural networks}

\author{Tim De Ryck\corref{mycorrespondingauthor}}
\cortext[mycorrespondingauthor]{Corresponding author}
\ead{tim.deryck@sam.math.ethz.com}
\author{Samuel Lanthaler}
\author{Siddhartha Mishra}
\address{Seminar for Applied Mathematics, ETH Z\"urich, R\"amistrasse 101, 8092 Z\"urich, Switzerland}

\begin{abstract}
We derive bounds on the error, in high-order Sobolev norms, incurred in the approximation of Sobolev-regular as well as analytic functions by neural networks with the hyperbolic tangent activation function. These bounds provide explicit estimates on the approximation error with respect to the size of the neural networks. We show that tanh neural networks with only two hidden layers suffice to approximate functions at comparable or better rates than much deeper ReLU neural networks.

\end{abstract}


\end{frontmatter}

\linenumbers

\section{Introduction}
Deep learning, relying on the use of deep artificial neural networks for regression and classification, has been very successful in different contexts in science and engineering in recent years \cite{lecun2015deep}. These include image recognition, natural language understanding, machine translation, game intelligence, robotics, autonomous systems and protein folding. 

Deep learning is also being increasingly used in scientific computing, particularly in the numerical solution of partial differential equations (PDEs). A very incomplete list of examples for the successful use of deep learning in this context includes the solution of high-dimensional linear and semi-linear parabolic partial differential equations \cite{HEJ1,E1} and references therein, the solution of parametric partial differential equations that arise in many-query problems like uncertainty quantification (UQ), PDE constrained optimization and (Bayesian) inverse problems  \cite{opschoor2019exponential,Kuty,PP1,LMR1,LMRP1} and in infinite-dimensional operator learning frameworks \cite{DeepOnets,LMK1,Stu1}. Another avenue for the application of deep neural networks in scientific computing is provided by \emph{physics-informed neural networks} (PINNs) \cite{Lag1,KAR1,KAR2,MM1,MM2}, which serve as replacements for traditional numerical methods for both forward as well as inverse problems for PDEs. 

The question of why deep neural networks are so successful at many diverse tasks in very different fields eludes a definitive answer. A very partial explanation may lie in the fact that artificial neural networks are \emph{universal approximators} i.e., any continuous (even measurable) mapping can be approximated by artificial neural networks to arbitrarily high accuracy \cite{barron1993universal,CY1,HOR1} and references therein. However, such universality results only imply the existence of a (shallow) neural network and do not provide any quantitative information (bounds) on the width of the underlying neural networks. 

The task of quantitatively relating the size and architecture of neural networks to their expressivity i.e., accuracy in approximating functions of a certain hypothesis class, has received considerable attention in the literature in the last few years. A seminal work in the direction is \cite{yarotsky2017error}, where the author derived explicit estimates on the size (width and depth) of a neural network with a ReLU activation function for approximating Lipschitz functions to any given accuracy in the $L^{\infty}$-norm. Expressivity results for such ReLU neural networks in Sobolev norms were presented in \cite{guhring2020error,herrmann2021constructive,opschoor2019exponential} and references therein, see also \cite{li2019better,montanelli2019new,opschoor2020deep,schwabzech2019,yarotsky2018optimal} and references therein for further approximation results for ReLU and related ReQU and RePU activation functions. 

Despite the fact that several quantitative results on the expressivity of neural networks have been obtained in recent years, we highlight some of the lacunae of the current state of the art in this direction, 
\begin{itemize}
    \item Most of the available results are on the expressivity and approximation properties of ReLU neural networks. Although ReLU activations are very common in practical applications of deep learning, there is a large number of areas where other activation functions are employed. One of the most popular activations is the tanh (hyperbolic tangent) activation function and the related sigmoid or logistic function (a scaled and shifted tanh). These activation functions are the basis of heavily used recurrent neural network (RNN) architectures such as LSTM \cite{lstm} and GRU \cite{GRU}. Other areas where smooth activation functions such as tanh are preferred over ReLU is in physics-informed neural networks (PINNs) for solving forward and inverse problems for PDEs \cite{KAR1,KAR2,MM1,MM2} and references therein, and in the use of quasi-random training points \cite{MR1,LMRS1}. Although the approximation abilities of general smooth activation functions have been investigated in \cite{costarelli2013approximation,constantine1996multivariate,guhring2021approximation,ohn2019smooth,pinkus1999approximation,siegel2020approximation}  and references therein, it is fair to say that the level of detail in existing results for the expressivity of ReLU neural networks, is not yet available for tanh neural networks. 
    \item Moreover, most of the approximation results for smooth activation functions, with the exception of the recent paper \cite{guhring2021approximation}, measure error in $L^p$-norms. However, it is essential to measure errors in higher-order Sobolev norms for many applications, such as PINNs where the neural network needs to be differentiated in order to evaluate the underlying PDE residual.
    \item A persistent focus of approximation results for neural networks has been to highlight the role of depth of the neural network, see \cite{poggio2017and} for a review and further references. In particular, {\color{black}{there are several results to the effect}} that very deep neural networks are, in some sense, more expressive than shallower networks, which in turn might explain the superior performance of deep neural networks in many applications {\color{black} \cite{rolnick2017power, lin2017does, bengio2007scaling, bianchini2014complexity}}. The empirical superiority of deep networks over their shallower counterparts has indeed been observed in many applications in computer science. However, in the context of scientific computing, empirical experience has revealed that shallower but wider networks result in superior performance over deep and narrow neural networks, see \cite{LSK1} and references therein. A reason for this observation lies in the fact that deeper networks might be harder to train in the relatively data poor regime of scientific computing. Some theoretical understanding of this deterioration of performance for deeper networks, at least in the context of ReLU networks is provided in \cite{grohs2021proof}.  However, most of the available approximation theory results trade width for depth and there is little theoretical understanding of why relatively shallow networks can perform well in some contexts. 
    \item Most of the available results on expressivity focus on asymptotic approximation rates i.e., the complexity of the network as the approximation error $\epsilon \rightarrow 0$. However, the fundamental question is how large a neural network should be to provide a certain accuracy of this approximation. This requires going beyond asymptotic approximation rates and providing explicit bounds on the underlying constants. {\color{black} Such explicit bounds are available for approximation in Hilbert and $L^p$ spaces with $p<\infty$ \cite{barron1993universal, barron1994approximation,kurkova2008geometric, lavretsky2002geometric, makovoz1996random, kainen2012dependence} and approximation using a non-standard activation function \cite{guliyev2018approximationa,guliyev2018approximationb}, but they remain mostly unavailable for function approximation in $W^{k,\infty}$ spaces with neural networks with widely used activation functions.} 
    \item The approximation error is only one component of the total error of neural networks, with optimization and generalization errors being the other components \cite{MLbook,CS1}. In particular, standard approaches to estimate the generalization error such as covering number estimates \cite{beck2019full} or Rademacher complexity \cite{MLbook} require explicit estimates on the weights of the underlying neural networks, in addition to bounds on their width and depth. Such estimates on weights of the best approximations of functions in the class of neural networks are rarely available in the current literature. 
\end{itemize}
The main objective of this paper is to address some of the afore-mentioned deficiencies in the literature on approximation properties of neural networks. We will focus on the expressivity of neural networks with the very popular tanh activation function and will aim to prove error and complexity bounds in high-order Sobolev norms for such tanh neural networks in approximating functions belonging to Sobolev spaces as well as $C^k$-spaces. We go beyond the usual practice of proving only asymptotic convergence rates and will provide explicit approximation error bounds for explicit network architectures in order to answer the question of \textit{``How large should a neural network be to approximate a specified function to some chosen accuracy $\epsilon>0$?"}. All our results will be for tanh neural networks with at most two hidden layers.  

{\color{black}A key application of our results is on the approximation of analytic functions by tanh neural networks. We will prove that a two hidden layer tanh neural network suffices to approximate an analytic function at an exponential rate, in terms of the network width, even in Sobolev norms.} This result provides an improvement over available results for the approximation of analytic functions by ReLU neural networks \cite{wang2018exponential,opschoor2019exponential,herrmann2021constructive} and also neural networks with smooth activation functions \cite{mhaskar1996neural} and further illustrate the powers of rather shallow tanh networks at approximating smooth functions. Finally, we also derive explicit bounds on the width of the tanh neural networks as well as asymptotic bounds on their weights, thus paving the way for bounds on the generalization error for these neural networks. 

The rest of the paper is organized as follows: in Section \ref{sec:2}, we introduce the notation for the rest of the paper. Our main results, presented in Section \ref{sec:5}, rely on the uniform approximation of polynomials by tanh neural networks, discussed in Section \ref{sec:3}, and on the approximation of a partition of unity, presented in Section \ref{sec:pou}. In Section \ref{sec:6}, we discuss the contents of this paper and distinguish them from other related papers. 

\section{Preliminaries}
\label{sec:2}

We start by providing an overview of all the notation and the definitions that will be used frequently throughout the paper. 

\subsection{Multi-index notation}\label{sec:multi-index}

For $d\in\mathbb{N}$, we call a $d$-tuple of non-negative integers $\alpha \in \N^d_0$ a multi-index. We write $\abs{\alpha} = \sum_{i=1}^d \alpha_i$, $\alpha! = \prod_{i=1}^d \alpha_i!$ and, for $x\in\mathbb{R}^d$, we denote by $x^\alpha =  \prod_{i=1}^d x_i^{\alpha_i}$ the corresponding multinomial. Given two multi-indices $\alpha, \beta\in \N^d_0$, we say that $\alpha \leq \beta$ if, and only if, $\alpha_i\leq \beta_i$ for all $i=1,\dots, d$. For a multi-index $\alpha$, we define the following multinomial coefficient
\begin{equation}
    \binom{\abs{\alpha}}{\alpha} = \frac{\abs{\alpha}!}{\alpha!}, 
\end{equation}
and, given $\alpha \le \beta$, we define a corresponding multinomial coefficient by
\begin{equation}
    \binom{\beta}{\alpha} = \prod_{i=1}^d \binom{\beta_i}{\alpha_i} = \frac{\beta !}{\alpha! (\beta-\alpha)!}.
\end{equation}
{\color{black}For $\Omega\subseteq \mathbb{R}^d$ and a function $f:\Omega\to\mathbb{R}$ we denote by
\begin{equation}
    D^\alpha f= \frac{\partial^{\abs{\alpha}} f}{\partial x_1^{\alpha_1}\cdots \partial x_d^{\alpha_d}}
\end{equation}
the classical or distributional (i.e. weak) derivative of $f$.}
We will frequently encounter the set $P_{n,d} = \{\alpha\in  \mathbb{N}_0^d : \abs{\alpha} = n\}$ (notation as in \cite{moak1990combinatorial}). In particular, we will need estimates on its cardinality. This is the subject of the following lemma. 

\begin{lemma}\label{lem:size-pqn}
Let $n\in\mathbb{N}$, $d\in\mathbb{N}_{\geq 2}$ and let $P_{n,d} = \{\alpha\in  \mathbb{N}_0^d : \abs{\alpha} = n\}$. Then 
\begin{equation}
    \abs{P_{n,d} } = \binom{n+d-1}{n} \leq \sqrt{\pi} \min\{e^{d-1}n^{d-1},e^{n}(d-1)^{n}\} \quad \text{and} \quad \abs{P_{d, d} } \leq 5 ^{d}. 
\end{equation}
\end{lemma}
\begin{proof}
It is well known that $\abs{P_{n,d} } = \binom{n+d-1}{n}$. We use Stirling's approximation,
\begin{align*}
     \abs{P_{n,d} } &=  \binom{n+d-1}{n} \leq \frac{e(n+d-1)^{n+d-1/2}}{2\pi n^{n+1/2}(d-1)^{d-1/2}} \\
     &\leq \frac{e }{2\pi} \left(\frac{n+d-1}{d-1}\right)^{d-1}\left(\frac{n+d-1}{n}\right)^{n}\sqrt{\frac{n+d-1}{n(d-1)}}\\
     &\leq \frac{e }{\sqrt{2}\pi} \left(1+\frac{n}{d-1}\right)^{d-1}\left(1+\frac{d-1}{n}\right)^{n}.
\end{align*}
To estimate the last term, we note that there are two possible approximations: for $a,b\geq 1$ it holds that $(1+a/b)^b\leq ea^b$ and also $(1+a/b)^b\leq e^a$. Using the fact that $e^2/\sqrt{2}\pi\leq\sqrt{\pi}$, we obtain
\begin{equation}
    \abs{P_{n,d} } \leq \sqrt{\pi}e^{d-1}n^{d-1} \qquad \text{and}\qquad \abs{P_{n,d} } \leq \sqrt{\pi}e^{n}(d-1)^{n}.
\end{equation}
Setting $n=d$ for $\lambda\in\mathbb{N}$, we also find that
\begin{equation}
    \abs{P_{d,d} } \leq\frac{e }{\sqrt{2}\pi} \left(1+\frac{d}{d-1}\right)^{d}\left(1+\frac{d-1}{d}\right)^{d} \leq \left(3+\frac{d}{d -1}\right)^{d} \leq 5^{d}, 
\end{equation}
since $\frac{e }{\sqrt{2}\pi} \leq 1$ and $\frac{x}{x-1}\leq 2$ for $x\geq 2$. 

\end{proof}

\subsection{Sobolev spaces}\label{sec:sobolev}

Let $d\in\mathbb{N}$, $1\leq p\leq \infty$ and let $\Omega \subseteq \mathbb{R}^d$ be open. We denote by $L^p(\Omega)$ the usual Lebesgue space and for $k\in\mathbb{N}_0$ we define the Sobolev space $W^{k,p}(\Omega)$ as
\begin{equation}
    W^{k,p}(\Omega) = \{f \in L^p(\Omega): D^\alpha f \in L^p(\Omega) \text{ for all } \alpha\in\mathbb{N}^d_0 \text{ with } \abs{\alpha}\leq k\}. 
\end{equation}
For $p<\infty$, we define the following seminorms on $W^{k,p}(\Omega)$, 
\begin{equation}
    \abs{f}_{W^{m,p}(\Omega)} = \left(\sum_{\abs{\alpha}= m}\norm{D^\alpha f}^p_{L^p(\Omega)}\right)^{1/p} \qquad \text{for } m=0,\ldots, k, 
\end{equation}
and for $p=\infty$ we define
\begin{equation}
    \abs{f}_{W^{m,\infty}(\Omega)} =\max_{\abs{\alpha}= m} \norm{D^\alpha f}_{L^\infty(\Omega)}\qquad \qquad \text{for } m=0,\ldots, k. 
\end{equation}
Based on these seminorms, we can define the following norm for $p<\infty$,
\begin{equation}
    \norm{f}_{W^{k,p}(\Omega)} = \left(\sum_{m=0}^k \abs{f}_{W^{m,p}(\Omega)}^p\right)^{1/p}, 
\end{equation}
and for $p=\infty$ we define the norm
\begin{equation}
    \norm{f}_{W^{k,\infty}(\Omega)} =\max_{0\leq m\leq k}  \abs{f}_{W^{m,\infty}(\Omega)}. 
\end{equation}
The space $W^{k,p}(\Omega)$ equipped with the norm $\norm{\cdot}_{W^{k,p}(\Omega)}$ is a Banach space. 
\subsection{Neural networks}\label{sec:neural-networks}

In this paper, we will consider function approximation using feedforward artificial neural networks where only connections between neighbouring layers are allowed. In the following, we formally introduce our definition of a neural network and the related terminology. 

Let $L\in\mathbb{N}$ and $l_0,\ldots, l_L\in\mathbb{N}$. Let $\sigma:\mathbb{R}\to\mathbb{R}$ be an \textit{activation function} and define the parameter space
\begin{equation}
    \Theta= \bigcup_{L\in\mathbb{N}}\bigcup_{l_0,\ldots,l_L\in\mathbb{N}}\bigtimes_{k=1}^L \left(\mathbb{R}^{l_k\times l_{k-1}}\times\mathbb{R}^{l_k}\right). 
\end{equation}
For $\theta\in\Theta$, we define $\theta_k:= (W_k,b_k)$ and $\mathcal{A}_k:\mathbb{R}^{l_{k-1}}\to\mathbb{R}^{l_{k}}:x\mapsto W_k x+b_k$ for $1\leq k\leq L$ and we denote by $\Psi_\theta:\mathbb{R}^{l_0}\to\mathbb{R}^{l_L}$, $x \mapsto \Psi_\theta(x)$, the function
\begin{equation}
    \Psi_\theta(x) = \begin{cases}\mathcal{A}_1(x) & L=1,\\ (\mathcal{A}_L \circ \sigma \circ \mathcal{A}_{L-1} \circ \sigma \circ\cdots \circ \sigma \circ \mathcal{A}_1)(x) & L\geq 2, \end{cases}
\end{equation}
where $\sigma$ is applied element-wise. 
We refer to $\Psi_\theta$ as the realization of the \textit{neural network} associated to the parameter $\theta$ with $L$ layers and widths $(l_0,l_1, \ldots, l_L)$. We refer to the first $L-1$ layers as \textit{hidden layers}. For $1\leq k\leq L$, we say that layer $k$ has width $l_k$ and we refer to $W_k$ and $b_k$ as the \textit{weights and biases} corresponding to layer $k$. {\color{black} The width of $\Psi_\theta$ is defined as $\max(l_0,\dots, l_L)$.} If $L=2$, we say that $\Psi_\theta$ is a \textit{shallow neural network}; if $L\geq 3$, we say that $\Psi_\theta$ is a \textit{deep neural network}. Hence, a shallow neural network has exactly one hidden layer whereas deep neural networks can have two or more hidden layers.  

In this work, we will focus on neural networks which use the hyperbolic tangent as activation function, defined by
\begin{equation}
    \sigma(x) := \tanh(x) = \frac{e^{x}-e^{-x}}{e^{x}+e^{-x}}\quad \text{for } x\in\mathbb{R}. 
\end{equation}
We will refer to these networks as \textit{tanh neural networks}. Even though our ideas can be carried over to other smooth activation functions, focusing on a particular activation will allow us to prove precise and explicit bounds without sacrificing the clarity of our arguments. In particular, we note that all results of this work directly apply to the sigmoid or logistic activation function, which is simply a shifted and scaled version of the hyperbolic tangent.

We close this section by recalling some basic properties of neural network calculus which we will use throughout, without explicitly referring to them. 
\begin{proposition}[Parallelization of neural networks] Let $L\in\mathbb{N}$, $l_0,l_0',\ldots, l_L,l_L'\in\mathbb{N}$ and $\theta, \vartheta\in\Theta$ such that $\Psi_\theta$ is a neural network with widths $(l_0, \ldots, l_L)$ and $\Psi_\vartheta$ is a neural network with widths $(l_0', \ldots, l_L')$. Then there exists $\eta\in\Theta$ such that $\Psi_\eta$ is a neural network with widths $(l_0+l_0', \ldots, l_L+l_L')$ for which it holds that $\Psi_\eta(x) = (\Psi_\theta((x_1, \ldots, x_{l_0})), \Psi_\vartheta((x_{l_0+1}, \ldots, x_{l_0+l_0'})))$ for all $x\in\mathbb{R}^{l_0+l_0'}$. 
\end{proposition}

\begin{proposition}[Composition of neural networks]
 Let $L,L'\in\mathbb{N}$, $l_0,\ldots, l_L=l_0',\ldots, l_{L'}'\in\mathbb{N}$ and $\theta, \vartheta\in\Theta$ such that $\Psi_\theta$ is a neural network with widths $(l_0, \ldots, l_L)$ and $\Psi_\vartheta$ is a neural network with widths $(l_0', \ldots, l_{L'}')$. Then there exists $\eta\in\Theta$ such that $\Psi_\eta$ is a neural network with widths $(l_0, \ldots, l_L=l_0', \ldots, l_{L'}')$ for which it holds that $\Psi_\eta = \Psi_\vartheta \circ \Psi_\theta$. 
\end{proposition}

\section{Uniform approximation of polynomials}\label{sec:3}
The first step in our strategy for deriving bounds on approximation error for tanh neural networks is to provide uniform bounds in Sobolev norms, on the error for approximating polynomials by shallow tanh neural networks. We do so in the current section.

The observation that a shallow neural network of fixed size can approximate monomials to arbitrary accuracy in the supremum norm was already observed in \cite{pinkus1999approximation}. A generalization of this approximation result to Sobolev norms was proven in e.g. \cite{guhring2021approximation}. In the current section, we present a \emph{novel} generalization that allows us to obtain explicit error estimates for the \textit{uniform} approximation of all polynomials of a certain maximal degree, which will be crucial for the efficient approximation of analytic functions. 

\subsection{Univariate polynomials}

We first describe how to approximate univariate polynomials of any degree with tanh neural networks. 
We introduce the $p$-th order central finite difference operator $\delta^p_h$ for any $f\in C^{p+2}([a,b])$ for some $p\in\mathbb{N}$ by
\begin{equation}\label{eq:central-fd}
    \delta^p_h[f](x) = \sum_{i=0}^p (-1)^i \binom{p}{i} f\left(x+\left(\frac{p}{2}-i\right)h\right).
\end{equation}
Next we define for any $p\in\mathbb{N}$,  $q\in2\mathbb{N}-1$ and $M>0$ the monomials $f_{p}:[-M,M]\to\mathbb{R}$ and the tanh neural networks $\hat{f}_{q,h}:[-M,M]\to\mathbb{R}$ as
\begin{equation}\label{eq:def-monomial}
    f_p(y):= y^p \qquad \text{and}\qquad \hat{f}_{q,h}(y) := \frac{\delta^q_{hy}[\sigma](0)}{\sigma^{(q)}(0)h^q}.
\end{equation}
We first prove that these neural networks are accurate approximations to monomials with odd degree.

\begin{lemma}\label{lem:uni-mon-sobolev}
Let $k\in\mathbb{N}_0$ and $s\in2\mathbb{N}-1$. Then it holds that for all $\epsilon>0$ there exists a shallow tanh neural network $\Psi_{s,{\color{black}\epsilon}}:[-M,M]\to \mathbb{R}^{\frac{s+1}{2}}$ of width $\frac{s+1}{2}$ such that
\begin{equation}
    \max_{\substack{p\leq s, \\p \textrm{ odd}}} \ck{f_p- (\Psi_{s,\epsilon})_{\color{black}\frac{p+1}{2}}} \leq \epsilon, 
\end{equation}
Moreover, the weights of $\Psi_{s,\epsilon}$ scale as $O\left(\epsilon^{-s/2}(2(s+2)\sqrt{2 M})^{s(s+3)}\right)$ for small $\epsilon$ and large $s$.
\end{lemma}

\begin{proof}
Let $p\leq s$ be odd and let $0<h<2/pM$. Let $0\leq m\leq \min\{k,p+1\}$. Then Taylor's theorem guarantees the existence of $\xi_{x,i}$ such that
\begin{align*}
    \begin{split}
        \frac{d^m}{dx^m} \delta^p_{hx}[\sigma](0) &= \sum_{i=0}^p (-1)^i \binom{p}{i}\left(\frac{p}{2}-i\right)^mh^m\cdot \sigma^{(m)}\left(\left(\frac{p}{2}-i\right)hx\right)\\
        &= \sum_{i=0}^p (-1)^i \binom{p}{i}\left(\frac{p}{2}-i\right)^mh^m \left(\sum_{l=m}^{p+1} \frac{\sigma^{(l)}(0)}{(l-m)!}\left(\frac{p}{2}-i\right)^{l-m}(hx)^{l-m}\right)\\
        &\quad + \sum_{i=0}^p (-1)^i \binom{p}{i}\left(\frac{p}{2}-i\right)^mh^m \frac{\sigma^{(p+2)}(\xi_{x,i})}{(p+2-m)!}\left(\frac{p}{2}-i\right)^{p+2-m}(hx)^{p+2-m}. 
    \end{split}
\end{align*}
From \cite[Theorem 1]{katsuura2009summations} it follows that 
\begin{equation} \label{eq:katsuura}
    \sum_{i=0}^p (-1)^i \binom{p}{i}\left(\frac{p}{2}-i\right)^l = p! \, \delta (l-p)
    =
    \begin{cases}
    p!, & (l=p), \\
    0, & (l\ne p).
    \end{cases}
\end{equation}
for $l=0,\ldots, p$. We observe that \eqref{eq:katsuura} remains true also for $l=p+1$, since all summands change sign when $i$ is replaced by $p-i$.
Using this fact, we can then rewrite the first term as 
\begin{align}
    \begin{split}
        &\sum_{i=0}^p (-1)^i \binom{p}{i}\left(\frac{p}{2}-i\right)^mh^m \left(\sum_{l=m}^{p+1} \frac{\sigma^{(l)}(0)}{(l-m)!}\left(\frac{p}{2}-i\right)^{l-m}(hx)^{l-m}\right)\\
        & \quad = h^m  \sum_{l=m}^{p+1} \frac{\sigma^{(l)}(0)}{(l-m)!}(hx)^{l-m} \sum_{i=0}^p (-1)^i \binom{p}{i}\left(\frac{p}{2}-i\right)^l\\
        & \quad =   \begin{rcases}
    \begin{dcases}
     h^m \frac{\sigma^{(p)}(0)}{(p-m)!}(hx)^{p-m} p!, &0\leq m\leq p \\
     0, & m=p+1
    \end{dcases}
  \end{rcases}
 = h^p \sigma^{(p)}(0) f_p^{(m)}(x). 
    \end{split}
\end{align}
Combining the previous results, it thus follows that we have
\begin{align*}
        \hat{f}^{(m)}_{p,h}(x) - f^{(m)}_p(x) &= \sum_{i=0}^p (-1)^i \binom{p}{i} \frac{1}{(p+2-m)!} \frac{\sigma^{(p+2)}(\xi_{x,i})}{\sigma^{(p)}(0)}\left(\frac{p}{2}-i\right)^{p+2}  h^{2} x^{p+2-m}. 
\end{align*}
Together with the lower and upper bounds on the derivatives of $\sigma$ from Lemma \ref{lem:tanh-derivatives} and Lemma \ref{lem:bound-der-tanh}, this yields for $m\le \min(k,p+1)$:
\begin{align}
    \begin{split}
        \abs{f_p- \hat{f}_{p,h}}_{W^{m,\infty}} &\leq \sum_{i=0}^p  \binom{p}{i} \frac{\abs{\sigma^{(p+2)}(\xi_{x,i})}}{\abs{\sigma^{(p)}(0)}}\abs{\frac{p}{2}-i}^{p+2}h^{2}M^{p+2}\\
        &\leq 2^{p} \frac{(2(p+2))^{p+3}}{1} \left(\frac{p}{2}\right)^{p+2}h^{2}M^{p+2}\\
        &\leq (2(p+2)pM)^{p+3} h^2.
    \end{split}
\end{align}
If $k\le p+1$, then this shows that 
\begin{align} \label{eq:klep}
\ck{f_p- \hat{f}_{p,h}} \le (2(p+2)pM)^{p+3} h^2.
\end{align}
If $k>p+1$, let $p+2 \leq m \leq k$. In this case, $f_p^{(m)}=0$, therefore it suffices to bound $\hat{f}_{p,h}^{(m)}$. We see that for $0<h<1$,
\begin{align}
    \begin{split}
         \abs{\hat{f}_{p,h}^{(m)}(x)} &= \abs{\frac{1}{h^p \sigma^{(p)}(0)}\sum_{i=0}^p (-1)^i \binom{p}{i}\left(\frac{p}{2}-i\right)^mh^m\cdot \sigma^{(m)}\left(\left(\frac{p}{2}-i\right)hx\right)}\\
         & \leq 2 \sum_{i=0}^p  \binom{p}{i}\abs{\frac{p}{2}-i}^mh^2 (2m)^{m+1} \\
         & \leq 2^{p+1} \left(\frac{p}{2}\right)^k (2k)^{k+1} h^2 \leq (2 p k)^{k+1} h^2. 
    \end{split}
\end{align}
We thus obtain, for arbitrary $k\in \N$:
\begin{equation}
     \ck{f_p- \hat{f}_{p,h}} \leq  \left((2(p+2)pM)^{p+3} + (2 p k)^{k+1}\right)h^2 =: \epsilon. 
\end{equation}

Furthermore observe that the weights scale as $O\left(\max_i\binom{p}{i}h^{-p}\right)$. For $\epsilon\to 0$ and large $p$, it holds that $O(h^{-p}) = O\left(\epsilon^{-p/2}((p+2)\sqrt{2 M})^{p(p+3)}\right)$, where the implied constant depends on $k$. Next, we find using Stirling's approximation that for $0\leq i \leq p$ it holds that
\begin{equation}
    \binom{p}{i} \leq \binom{p}{\frac{p-1}{2}} \leq \frac{ep^{p+1/2}}{2\pi \left(\frac{p-1}{2}\right)^{\frac{p}{2}}\left(\frac{p+1}{2}\right)^{\frac{p}{2}+1}} = O\left(\frac{2^p}{\sqrt{p}}\right). 
\end{equation}
The weights therefore scale as $O\left(\epsilon^{-p/2}(2(p+2)\sqrt{2 M})^{p(p+3)}\right)$. 

Regarding the network architecture, note that the neurons needed for all $\hat{f}_{p,h}$ are already available in the network $\hat{f}_{s,h}$. This allows us to define the shallow tanh neural network $\Psi_{s,\epsilon}$ by $(\Psi_{s,\epsilon})_p = \hat{f}_{p,h}$ such that it only has $\frac{s+1}{2}$ neurons in its hidden layer. The width follows directly from its definition and the fact that $\sigma$ is an odd function.
\end{proof}

We would like to state that the above proof is largely inspired by \cite[Proposition 4.7]{guhring2021approximation} but differs at some crucial points. In particular, we take into account the fact that $\xi_{x,i}$'s are functions of $x$ and derivatives with respect to $x$ have to take this into account. 

We now extend the previous result to monomials with even degree. To this end, we rely on the observation that for $n\in\mathbb{N}$ and $\alpha>0$, it holds that
\begin{equation} \label{eq:recursion-exact}
    \begin{aligned}
    y^{2n} = \frac{1}{2\alpha(2n+1)} 
    &\Bigg(
    (y+\alpha)^{2n+1}-(y-\alpha)^{2n+1} 
    \\
    &\qquad - 2\sum_{k=0}^{n-1} \binom{2n+1}{2k} \alpha^{2(n-k)+1}y^{2k}
    \Bigg). 
    \end{aligned}  
\end{equation}
This formula allows us to construct recursively defined tanh neural network approximations of even powers of $y$. The following lemma quantifies the uniform approximation accuracy of these networks in the Sobolev norm. 

\begin{lemma}\label{lem:monomials}
Let $k\in\mathbb{N}_0, s\in 2\mathbb{N}-1$ and ${\color{black}M>0}$. For every $\epsilon>0$, there exists a shallow tanh neural network $\psi_{s,{\color{black}\epsilon}} : {\color{black}[-M,M]}\to \mathbb{R}^s$ of width $\frac{3(s+1)}{2}$ such that 
\begin{equation}
    \max_{p\leq s} \ck{f_p- (\psi_{s,\epsilon})_p} \leq\epsilon.
\end{equation}
Furthermore, the weights scale as $O\left(\epsilon^{-s/2}(\sqrt{M}(s+2))^{3s(s+3)/2}\right)$ for small $\epsilon$ and large $s$.
\end{lemma}
\begin{proof}
For $h,M>0$, $p\leq s$, we define $\hat{f}_{p,h}:[-M-1,M+1]\to\mathbb{R}$ as in \eqref{eq:def-monomial}. 
For $\epsilon>0$, $h>0$ small enough and dependent on $\epsilon$, $\alpha\leq 1$ and $y\in [-M,M]$, we define $(\psi_{s,\epsilon}(y))_p =  \hat{f}_{p,h}(y)$ for $p$ odd and, for $p=2n$ even, we define $(\psi_{s,\epsilon}(y))_p = (\psi_{s,\epsilon}(y))_{2n}$ recursively by $(\psi_{s,\epsilon})_0(y) := 1$, and 
\begin{align}\label{eq:even-power}
\begin{aligned}
   (\psi_{s,\epsilon}(y))_{2n} = 
     \frac{1}{2\alpha(2n+1)} 
     &\Bigg(
     \hat{f}_{2n+1,h}(y+\alpha)-\hat{f}_{2n+1,h}(y-\alpha) \\
     &\qquad 
     - 2\sum_{k=0}^{n-1} \binom{2n+1}{2k} \alpha^{2(n-k)+1}(\psi_{s,\epsilon}(y))_{2k}
     \Bigg).
     \end{aligned}  
\end{align}
Moreover, we introduce the notation $E_p = \ck{f_p-(\psi_{s,\epsilon})_p}$. We will prove the statement that for all $\epsilon>0$, there exists $h>0$, such that for all $p\leq s$, we have 
\begin{equation}\label{eq:statement-even}
     E_p \leq E^\ast_p := \frac{2^{p/2}(1+\alpha)^{(p^2+p)/2}}{\alpha^{p/2}} \cdot\epsilon.
\end{equation}
We first note that choosing $h$ as in Lemma \ref{lem:uni-mon-sobolev} implies that 
\begin{equation} \label{eq:odd}
    \max_{\substack{p\leq s, \\p \text{ odd}}} E_p \leq \epsilon,
\end{equation}
which proves the statement for $p$ odd, since $(1+\alpha)/\alpha\geq 1$. We will now prove \eqref{eq:statement-even} for even $p$ using induction. First note that
\begin{equation}
    E_2 \leq \frac{1}{6\alpha} \cdot 2 \epsilon \le E^\ast_2, 
\end{equation}
which proves the base step. To prove the induction step, let $n\in\mathbb{N}$ be such that $2n+1\leq s$ and $n>1$, and we assume by the induction hypothesis that $E_{2k} \le E^\ast_{2k}$ for all $k < n$. It then follows from \eqref{eq:recursion-exact} and \eqref{eq:even-power}, that
\begin{equation}\label{eq:E_2N}
    E_{2n} \leq \frac{1}{2\alpha(2n+1)} \left( E_{2n+1} + E_{2n+1} + 2\sum_{k=1}^{n-1} \binom{2n+1}{2k} \alpha^{2(n-k)+1}E_{2k}\right) . 
\end{equation}
Note that by the induction hypothesis and the fact that $E^\ast_{2k}$ is monotonically increasing in $k$, we have $E_{2k} \le E^\ast_{2k} \le E_{2(n-1)}^\ast$. Using also \eqref{eq:odd}, and the fact that $\epsilon \le E^\ast_{2(n-1)}$,  this allows us to estimate \eqref{eq:E_2N}, by
\begin{align}
    \begin{split}
         E_{2n} &\leq \frac{1}{\alpha(2n+1)} \left( \max_{\substack{p\leq s, \\p \text{ odd}}} E_p + \sum_{k=1}^{n-1} \binom{2n+1}{2k} \alpha^{2(n-k)+1}E_{2(n-1)}^*\right)\\
         &\leq \frac{1}{\alpha} \left( E_{2(n-1)}^* + (1+\alpha)^{2n+1} E_{2(n-1)}^*\right)\\
         &\leq \frac{2}{\alpha}(1+\alpha)^{2n+1} E_{2(n-1)}^*. 
    \end{split}
\end{align}
Recalling the definition of $E^\ast_{2(n-1)}$, we obtain
\begin{equation}
    E_{2n} \leq \frac{2}{\alpha}(1+\alpha)^{2n+1} E_{2(n-1)}^* \le \left( \frac{2}{\alpha}(1+\alpha)^{2n+1}\right)^{n} \cdot \epsilon = E^\ast_{2n}. 
\end{equation}
This proves the claimed estimate \eqref{eq:statement-even} also for the case where $p=2n$ is even, and therefore concludes the proof of \eqref{eq:statement-even}. 

Next, we optimize \eqref{eq:statement-even} by choosing the optimal value of $\alpha$. Lemma \ref{lem:alpha} proves that the optimal choice is $\alpha=1/s$. We conclude that for any $\epsilon> 0$, there exists a shallow tanh neural network $\psi_{s,\epsilon}$, with width independent of $\epsilon$, such that
\begin{equation}
    \max_{p\leq s} \ck{f_p-(\psi_{s,\epsilon})_p} \leq \sqrt{e} (2es)^{s/2} \epsilon. 
\end{equation}
Replacing $\epsilon \to \epsilon /  \sqrt{e} (2es)^{s/2}$ recovers the claimed error bound in the statement of this lemma. To quantify the size of the weights, we observe that equation \eqref{eq:even-power} reveals that the weight bound of Lemma \ref{lem:uni-mon-sobolev} needs to be multiplied with a factor
\begin{equation}
   \max_k s \binom{s}{2k} \left(\frac{1}{s}\right)^{s-2k} \leq s \sum_{j=0}^s \binom{s}{j} \left(\frac{1}{s}\right)^{s-j}\leq  s\left(1+\frac{1}{s}\right)^s = O(s), 
\end{equation}
where we used the binomial theorem. The weight bound can seen to be equal to
\begin{equation}
    O\left(\epsilon^{-s/2}s(2\sqrt[4]{2es}\sqrt{2 (M+1)}(s+2))^{s(s+3)}\right) = O\left(\epsilon^{-s/2}(\sqrt{M}(s+2))^{3s(s+3)/2}\right)
\end{equation}
for small $\epsilon$ and large $s$.
This proves the weight bound stated in the lemma. 

Finally, we note that the constructed approximations indeed correspond to a shallow tanh neural network of the stated size. Indeed, one can see from the fact that $\sigma$ is odd, equation \eqref{eq:central-fd}, Lemma \ref{lem:uni-mon-sobolev} and equation \eqref{eq:even-power} that a shallow tanh neural network suffices, where the values of the $3(s+1)/2$ neurons in the hidden layer are given by
\begin{equation}
    \sigma\left(\left(\frac{s}{2}-i\right)h(y+\beta)\right)\qquad \text{where } i=0, 1, \ldots \frac{s-1}{2}  \text{ and } \beta\in\{-\alpha, 0, \alpha\}. 
\end{equation}

\end{proof}

\begin{remark}
Combined with the Weierstrass approximation theorem \cite{weierstrass1885analytische}, the preceding results show that any continuous function can be uniformly approximated in supremum norm on a compact interval by shallow tanh neural networks to arbitrary accuracy, as was already observed by e.g. \cite{pinkus1999approximation}. Using the constructive proof of the Weierstrass approximation theorem based on Bernstein polynomials, one can even obtain a rate of convergence in terms of the width of the neural network and the modulus of continuity of the continuous function \cite{davidson2009real}. 
\end{remark}

\begin{remark}\label{rem:why-odd}
Note that one can also construct monomials with even powers directly, as was done in e.g. \cite{guhring2021approximation}. Indeed, there exists an $x\in \mathbb{R}$ such that $\tanh^{(p)}(x) \neq 0$ for all $p\in \mathbb{N}$, allowing us to use a neural network as in \eqref{eq:def-monomial} for even $p$ as well. However, a key difficulty lies in explicitly finding a function $\gamma:\mathbb{N}\to (0,\infty)$ such that $\abs{\tanh^{(p)}(x)} \geq \gamma(p)$ for all $p\in \mathbb{N}$. It is unclear if such a function, which is quite essential when proving uniform bounds as in Lemma \ref{lem:monomials}, can be constructed directly. Instead, our construction circumvents this issue and can be readily extended to other activation functions. 
\end{remark}

\subsection{Approximating multivariate polynomials}
Next, we consider the approximation of multivariate polynomials using tanh neural networks. As an application, we will also present two different approximations of the multiplication operator. 

First, recall the set $P_{n,q} = \{\alpha\in  \mathbb{N}_0^q : \abs{\alpha} = n\}$ from Section \ref{sec:multi-index}. The multinomial theorem implies that for $\alpha \in P_{n,q}$ and $\omega \in \mathbb{R}^q$, it holds that
\begin{equation}\label{eq:multinomial}
     \sum_{\beta\in P_{n,q}} \binom{n}{\beta} \frac{\alpha^\beta}{n^n} \omega^\beta = \left(\sum_{i=1}^q \frac{\alpha_i}{n}\omega_i\right)^n. 
\end{equation}
Now let $x\in \mathbb{R}^d$, set $q=d$ and $\omega = x$. Then the set $\{\omega^\beta: \beta\in P_{n,q}\}$ corresponds to the set of all $d$-variate monomials of total degree \textit{equal to} $n$. Similarly, one can set $q=d+1$ and $\omega=(1,x)$, such that $\{\omega^\beta: \beta\in P_{n,q}\}$ corresponds to the set of all $d$-variate monomials of total degree \textit{at most} $n$. It is the goal of this section to approximate these monomials $\omega^\beta$ using tanh neural networks. Notice however that the results from the previous section already allow us to approximate the right hand side of \eqref{eq:multinomial}, as it is merely a composition of a linear map and a univariate monomial. Writing $b_\alpha = \left(\sum_i \alpha_i\omega_i/n\right)^n$ one can interpret \eqref{eq:multinomial} for every $\alpha \in P_{n,q}$ as the linear equation $\sum_\beta D_{\alpha,\beta} \omega^\beta = b_\alpha$, where
\begin{equation}\label{eq:dyson}
    D_{\alpha,\beta} = \binom{n}{\beta} \frac{\alpha^\beta}{n^n}, 
\end{equation}
which leads us to a linear system $\{\sum_\beta D_{\alpha,\beta} \omega^\beta = b_\alpha : \alpha \in P_{n,q}\}$ with as unknowns the monomials $\omega^\beta$. Since the \textit{Dyson matrix} $D = (D_{\alpha,\beta})_{\alpha,\beta\in P_{n,q}}$, where the order of rows and columns reflects the lexicographic order on $P_{n,q}$, is invertible \cite{moak1990combinatorial}, it is possible to write every monomial as a linear combination of the $b_\alpha$'s. We will exploit this fact to construct approximations of multivariate polynomials and the multiplication $\prod_{i=1}^d x_i$ in particular. 

\begin{lemma}\label{lem:monomial-sobolev}
Let $q,n\in \mathbb{N}$, $k\in \mathbb{N}_0$ and $M>0$. Then for every $\epsilon>0$, there exists a shallow tanh neural network $\Psi_{n,q}: [-M,M]^q\to\mathbb{R}^{\abs{P_{n,q}}}$ of width $3\left\lceil\frac{n+1}{2}\right\rceil\abs{P_{n,q} }$ such that
\begin{equation}
   \max_{\beta\in P_{n,q}} \ck{\omega^\beta - (\Psi_{n,q}(\omega))_{\iota(\beta)}} \leq \epsilon,
\end{equation}
{\color{black}where $\iota:P_{n,q}\to\{1, \ldots \abs{P_{n,q}}\}$ is a bijection.} Furthermore, the weights of the network scale as $O\left(\epsilon^{-n/2}(n(n+2))^{3(n+2)^2}\right)$ for small $\epsilon$ and large $n$.
\end{lemma}
\begin{proof}
From the previous section, we can see that approximating $b_\alpha = \left(\sum_i \alpha_i\omega_i/n\right)^n$ requires a shallow tanh subnetwork $\widehat{b_\alpha}$ of width $3\left\lceil\frac{n+1}{2}\right\rceil$. As we require $\abs{P_{n,q} }$ such subnetworks, the total network width can be summarized as $3\left\lceil\frac{n+1}{2}\right\rceil\abs{P_{n,q} }$. Now denote by $\widehat{\omega^\beta}$ the neural network approximation one obtains by solving the linear system $\{\sum_\beta D_{\alpha,\beta} \widehat{\omega^\beta} = \widehat{b_\alpha} : \alpha \in P_{n,q}\}$. We then set $(\Psi_{n,q}(\omega))_{\iota(\beta)} :=(\widehat{\omega^\beta})_\beta$, where $\iota:P_{n,q}\to\{1, \ldots \abs{P_{n,q}}\}$ is a bijection.  Then it holds that
\begin{equation}
    \ck{(\widehat{\omega^\beta})_\beta-({\omega^\beta})_\beta} \leq \norm{D^{-1}}_\infty \ck{(\widehat{b_\alpha})_\alpha-(b_\alpha)_\alpha}. 
\end{equation}
Now define $h_\alpha(\omega) = \sum_i \alpha_i\omega_i/n$, then it holds that $\widehat{b_\alpha} - b_\alpha = ((\psi_{s,\epsilon})_n -f_n) \circ h_\alpha$, where $\psi_{s,\epsilon}$ is as in Lemma \ref{lem:monomials} and $s=2\left\lfloor \frac{n}{2}\right\rfloor+1$. It is easy to check that $\ck{h_\alpha}\leq \max\{1,M\}$. 
Invoking Lemma \ref{lem:faa-di-bruno} then gives us
\begin{equation}
     \ck{\widehat{b_\alpha} - b_\alpha} \leq   16(e^2k^{4}q^2)^{k} \ck{(\psi_{s,\epsilon})_n -f_n} \max\{1,M\}^k.
\end{equation}
In addition, Lemma \ref{lem:dysoninv} provides us with the bound
\begin{equation}
    \norm{D^{-1}}_\infty \leq (n!)^3 \abs{P_{n,q}}^2 2^n \leq \pi e^3 q^{2n} n^{3(n+1/2)},
\end{equation}
where we used Stirling's approximation and Lemma \ref{lem:size-pqn}. 
Now let $\epsilon>0$. Combining the two obtained inequalities with Lemma \ref{lem:monomials} then proves that
\begin{equation}
    \ck{(\widehat{\omega^\beta})_\beta-({\omega^\beta})_\beta} \leq \norm{D^{-1}}_\infty  \cdot 16(e^2k^{4}q^2)^{k} \max\{1,M\}^k \cdot \epsilon
\end{equation}
where the weights of $(\psi_{s,\epsilon})_n$ scale as $O\left(\epsilon^{-n/2}(\sqrt{M}(n+2))^{3n(n+3)/2}\right)$ for small $\epsilon$ and large $n$. We can now rescale $\epsilon$ such that
\begin{equation}
    \ck{(\widehat{\omega^\beta})_\beta-({\omega^\beta})_\beta} \leq \epsilon. 
\end{equation}
As a consequence, the weights of $\Psi_{n,q}$ will scale as
\begin{equation}
    O\left(\epsilon^{-n/2}\norm{D^{-1}}_\infty^{n/2+1}\left(4(ek^2q)^k\sqrt{1+M}\right)^n(\sqrt{M}(n+2))^{3n(n+2)/2}\right). 
\end{equation}
Note that $O\left(\norm{D^{-1}}_\infty^{n/2+1}\right) = O\left((\pi e^3 qn)^{3(n+2)^2/2}\right)$ and that therefore a (conservative) upper bound of the weights of $\Psi_{n,q}$ is given by
\begin{equation}
    O\left(\epsilon^{-n/2}(n(n+2))^{3(n+2)^2}\right)
\end{equation}
for small $\epsilon$ and large $n$. 
\end{proof}

\begin{corollary}[Approximation of multivariate monomials]\label{cor:mon-mult1}
Let $d,s\in \mathbb{N}$, $k\in \mathbb{N}_0$ and $M>0$. Then for every $\epsilon>0$, there exists a shallow tanh neural network $\Phi_{s,d}:[-M,M]^d\to\mathbb{R}^{\abs{P_{s,d+1}}}$ of width $3\left\lceil\frac{s+1}{2}\right\rceil\abs{P_{s,d+1} }$ such that
\begin{equation}
   \max_{\beta\in P_{s,d+1}} \ck{x^\beta - (\Phi_{s,d}(x))_{\iota(\beta)}} \leq \epsilon,
\end{equation}
{\color{black}where $\iota:P_{s,d+1}\to\{1, \ldots \abs{P_{s,d+1}}\}$ is a bijection}. Furthermore, the weights of the network scale as $O\left(\epsilon^{-s/2}(s(s+2))^{3(s+2)^2}\right)$ for small $\epsilon$ and large $s$.
\end{corollary}
\begin{proof}
The statement follows directly from Lemma \ref{lem:monomial-sobolev} with $n \leftarrow s$, $q \leftarrow d+1$ and $\omega \leftarrow (1,x)$, where $x\in[-M,M]^d$. 
\end{proof}

Next, we discuss how the multiplication operator can be approximated. To begin with, Lemma \ref{lem:monomial-sobolev} shows that the multiplication of $d$ numbers can easily be approximated using a shallow tanh neural network. 

\begin{corollary}[Shallow approximation of multiplication of $d$ numbers]\label{cor:mult-shallow}
Let $d\in \mathbb{N}$, $k\in \mathbb{N}_0$ and $M>0$. Then for every $\epsilon>0$, there exists a shallow tanh neural network $\widehat{\times}_d^\epsilon: [-M,M]^d\to\mathbb{R}$ of width $3\left\lceil\frac{d+1}{2}\right\rceil\abs{P_{d,d} }$ such that
\begin{equation}
   \ck{\widehat{\times}_d^\epsilon(x)-\prod_{i=1}^d x_i} \leq \epsilon.
\end{equation}
Furthermore, the weights of the network scale as $O(\epsilon^{-d/2})$ for small $\epsilon$.
\end{corollary}
\begin{proof}
The statement follows directly from Lemma \ref{lem:monomial-sobolev} with $n \leftarrow d$, $q \leftarrow d$ and $\omega \leftarrow x$, where $x\in[-M,M]^d$.
\end{proof}

One issue with this shallow approximation is that the width of the network grows quickly with the dimension. The next lemma shows that the same accuracy can also be obtained using a deep tanh neural network for which both width and depth scale at most linearly with the input dimension. 

\begin{lemma}[Deep approximation of multiplication of $d$ numbers]\label{lem:mult-deep}
Let $d\in \mathbb{N}$, $k\in \mathbb{N}_0$ and $M>0$. Then for every $\epsilon>0$, there exists a tanh neural network $\widehat{\times}_d^\epsilon: [-M,M]^d\to\mathbb{R}$ with $\lceil \log_2(d)\rceil$ hidden layers and of width at most $3d$ such that
\begin{equation}
   \ck{\widehat{\times}_d^\epsilon(x)-\prod_{i=1}^d x_i} \leq \epsilon.
\end{equation}
Furthermore, the weights of the network scale as $O(\epsilon^{-1/2})$ for small $\epsilon$.
\end{lemma}

\begin{proof}
Using the finite difference approach \eqref{eq:central-fd}, we can approximate the quadratic function using $\delta^2_h[f](x_0)$ for some $h>0$ and $x_0\in[-1,1]$ such that $\sigma^{(2)}(x_0)\neq 0$. Observing that
\begin{equation}\label{eq:xy}
    xy = \frac{1}{4}\left((x+y)^2-(x-y)^2\right)
\end{equation}
then provides a recipe to approximate (in Sobolev norm) the multiplication of two numbers using a shallow tanh neural network with 6 neurons in its hidden layer. The proof is similar to that of Lemma \ref{lem:uni-mon-sobolev}. Moreover, Lemma \ref{lem:uni-mon-sobolev} shows as well that the identity can be approximated using a shallow tanh neural network with only one neuron in its hidden layer. 

The multiplication of $d$ numbers then follows easily from the multiplication of $2$ numbers. In e.g. \cite[Proposition 2.36]{opschoor2019exponential}, it is proven that the multiplication of $d$ numbers requires a neural network in the form of a binary tree of depth $\lceil \log_2(d)\rceil$ where each node computes the (approximate) multiplication of two numbers. The proof of our error bound follows from Lemma \ref{lem:leibniz} and \ref{lem:faa-di-bruno}. 
\end{proof}

\begin{remark}
For simplicity and motivated by its widespread use, we only focused on the hyperbolic tangent activation function here. Our approach can be generalized to any activation function $\phi$ for which there exist $\mathcal{P}\subseteq \mathbb{N}$ with $\sup \mathcal{P} = \infty$ and an explicitly known function $\gamma:\mathcal{P}\to (0,\infty)$ with $\abs{\phi^{(p)}} \geq \gamma(p)$ for all $p\in \mathcal{P}$. Monomials with degree $p\in \mathcal{P}$ can be constructed as in \eqref{eq:def-monomial}, the construction of monomials with degree $p\in \mathbb{N}\setminus\mathcal{P}$ is similar to the one described for multivariate polynomials. 
\end{remark}

\section{Approximation of partition of unity}\label{sec:pou}

Once we have approximated polynomials with shallow tanh neural networks, the next step in our construction is to approximate a suitable partition of unity. In this section, we show how one can mimic a partition of unity using tanh neural networks. We recall that a partition of unity is a set of functions $f_i:[0,1]^d\to[0,1]$ such that every $f_i$ is non-zero on only a small part of $[0,1]^d$ and such that $\sum_i f_i = 1$. For ReLU and RePU neural networks, such partitions of unity can be constructed exactly \cite{yarotsky2017error}. For tanh neural networks, we will prove that an approximate partition of unity can be constructed. A unifying framework for approximating partitions of unity by general neural networks has been proposed in \cite{guhring2021approximation}.

Let $d,N\in\mathbb{N}$ and $k\in\mathbb{N}_0$. For every $j\in\mathbb{N}^d$ with $\norm{j}_\infty \leq N$ we define $x^N_j$ such that $(x^N_j)_i = j_i/N_i$. We also define
\begin{equation}
    I_j^N = \bigtimes_{i=1}^d \left((j_i-1)/N,j_i/N\right). 
\end{equation}
Let $R >0$ be such that {\color{black}$\vert\sigma^{(m)}\vert$} is decreasing on $[R,\infty)$ for every $1\leq m\leq k$. Given $\epsilon > 0$, we first find an $\alpha = \alpha(N,\epsilon)$ large enough such that
\begin{align}\label{eq:alpha}
    \alpha/N \geq R, \quad 1 - \sigma(\alpha/N) \leq \epsilon, \quad \alpha^m \abs{\sigma^{(m)}(\alpha/N)} \leq \epsilon \text{  for all } 1\leq m \leq k. 
\end{align}
This is possible because $\lim_{x\to\infty} \sigma(x)= 1$ and because of Lemma \ref{lem:bound-der-tanh}. In particular, Lemma \ref{lem:alpha-growth} shows that a suitable choice of $\alpha$ is given by 
\begin{equation}
   \alpha = N \max\left\{R,\ln(\frac{(2k)^{k+1}(Nk)^k}{e^k\epsilon})\right\}.
\end{equation}
For $y\in \mathbb{R}$, we then define
\begin{align}
    \rho_1^N(y) &= \frac{1}{2}-\frac{1}{2}\sigma\left(\alpha\left(y-\frac{1}{N}\right)\right),\\
    \rho_j^N(y) &= \frac{1}{2}\sigma\left(\alpha\left(y-\frac{j-1}{N}\right)\right) - \frac{1}{2}\sigma\left(\alpha\left(y-\frac{j}{N}\right)\right)\quad \text{for } 2\leq j \leq N-1,\\
    \rho_N^N(y) &= \frac{1}{2}\sigma\left(\alpha\left(y-\frac{N-1}{N}\right)\right)+\frac{1}{2}.
\end{align}
In the remainder of the paper, we will assume for simplicity that $\rho_j^N$ is always of the second form. The calculations involving $\rho_1^N$ and $\rho_N^N$ can be done entirely similarly and do not change the stated results. 
Finally, we define for $D\leq d$ the functions
\begin{equation}
    \Phi^{N,D}_j(x) = \prod^{D}_{i=1} \rho_{j_i}^{N_i}(x_i)
\end{equation}
and the sets $\mathcal{V}_D = \{v\in\mathbb{Z}^d: \max_{1\leq i \leq D}\abs{v_i}\leq 1 \text{ and } v_{D+1}=\cdots = v_d = 0\}$. We will prove that the functions $\Phi^{N,d}_j$ approximate a partition of unity in the sense that for every $j$ it holds on $I_j^N$ that,
\begin{equation}
    \sum_{v\in\mathcal{V}_d}\Phi^{N,d}_{j+v} \approx 1 \quad \text{and} \quad \sum_{\substack{v\not\in\mathcal{V}_d,\\ j+v \in \{1,\ldots, N\}^d}}\Phi^{N,d}_{j+v} \approx 0.
\end{equation}
An example for $d=1$ and $N=7$ is shown in Figure \ref{fig:pou}. The next two lemmas formalize this approximation. Finally, a tanh neural network approximation of $\Phi^{N,d}_j$ can be constructed by replacing the multiplication operator by the network from e.g. Corollary \ref{cor:mult-shallow} or Lemma \ref{lem:mult-deep}.

\begin{figure}
    \centering
    
    \begin{tikzpicture}
\draw[->] (-1,0) -- (8,0);
\draw[->] (0,-1) -- (0,3);
\def\a{1.9}
\draw[domain=0:3,smooth,variable=\x,red] plot ({\x}, 
 {1-tanh(\a*(\x-1))});
 \draw[domain=0:4,smooth,variable=\x,red] plot ({\x}, 
 {tanh(\a*(\x-1))-tanh(\a*(\x-2))});
 \draw[domain=0:5,smooth,variable=\x,blue] plot ({\x}, 
 {tanh(\a*(\x-2))-tanh(\a*(\x-3))});
 \draw[domain=1:6,smooth,variable=\x,blue] plot ({\x}, 
 {tanh(\a*(\x-3))-tanh(\a*(\x-4))});
\draw[domain=2:7,smooth,variable=\x,blue] plot ({\x}, 
 {tanh(\a*(\x-4))-tanh(\a*(\x-5))});
 \draw[domain=3:7,smooth,variable=\x,red] plot ({\x}, 
 {tanh(\a*(\x-5))-tanh(\a*(\x-6))});
 \draw[domain=4:7,smooth,variable=\x,red] plot ({\x}, 
 {tanh(\a*(\x-6))+1});
 
 \draw[domain=0:7,smooth,variable=\x,blue,thick] plot ({\x}, 
 {tanh(\a*(\x-2))-tanh(\a*(\x-5))});
 \draw[domain=0:7,smooth,variable=\x,red, thick] plot ({\x}, 
 {2-tanh(\a*(\x-2))+tanh(\a*(\x-5))});
 \draw[domain=0:7,dashed,variable=\x,black] plot ({\x}, 2);
 
 \node[left] at (0,2) {$1$};
 \node[below] at (-0.2,0) {$0$};
 \foreach \i in {1,...,6}
{
         \node[below] at (\i,0) {$\frac{\i}{7}$};
}
\node[below] at (7,0) {$1$};
\node[below,gray] at (3.5,-0.4) {$I_4^7$};
\node[above,blue] at (4.5,2) {\small $\sum_{v\in\mathcal{V}_1}\Phi^{7,1}_{5+v}$};
\node[above,red] at (1.5,2) {\small $\sum_{v\not\in\mathcal{V}_1,\: 1 \leq 4+v \leq 7}\Phi^{7,1}_{4+v}$};
 
 \fill [gray, opacity=0.2] (3,-1) rectangle (4,3);
 
\end{tikzpicture}
    
    \caption{Example of an approximate partition of unity on $[0,1]$ with $N=7$. The thin lines represent the $\Phi^{7,1}_{j} = \rho^7_j$, $1\leq j \leq 7$.
    }
    \label{fig:pou}
\end{figure}
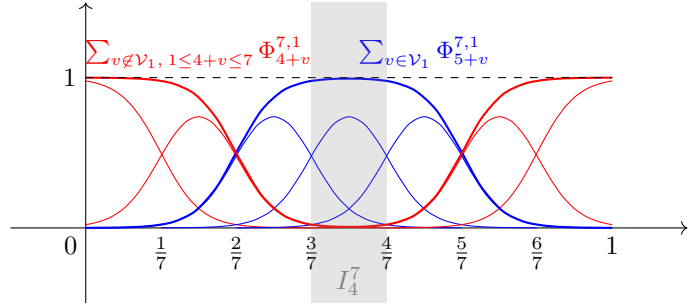

\begin{lemma}\label{lem:indicator-close}
If $0<\epsilon < 1/4$, then 
\begin{equation}
     \norm{\sum_{v\in \mathcal{V}_d}\Phi^{N,d}_{j+v}-1}_{W^{k,\infty}(I_j^N)} \leq 2^{dk} d\epsilon.
\end{equation}
\end{lemma}
\begin{proof}
We will prove the statement holds by induction on $d$. We first note that, for $d=1$, we have
\begin{align} \label{eq:V1}
\begin{aligned}
    \sum_{v\in \mathcal{V}_1}\Phi^{N,1}_{j+v}(x) 
    &= \sum_{l=-1}^1 \rho_{j_1+l}^N(x_1)
    \\
    &=\frac{1}{2}\sigma\left(\alpha\left(x_1-\frac{j_1-2}{N}\right)\right) - \frac{1}{2}\sigma\left(\alpha\left(x_1-\frac{j_1+1}{N}\right)\right),
\end{aligned}   
\end{align}
from which easily follows that
\begin{equation}
    \sum_{v\in \mathcal{V}_1}\Phi^{N,1}_{j+v}(x) \leq 1. 
\end{equation}
Next, note that for $x\in I_j^N$
\begin{align}
\begin{split}
\sum_{v\in \mathcal{V}_1}\Phi^{N,1}_{j+v}(x) &=  \frac{1}{2}\sigma\left(\alpha\left(x_1-\frac{j_1-2}{N}\right)\right) - \frac{1}{2}\sigma\left(\alpha\left(x_1-\frac{j_1+1}{N}\right)\right)\\
&\geq \sigma\left(\frac{\alpha}{N}\right) \geq  1-\epsilon,
\end{split}
\end{align}
where we used the definition of $\alpha$ on the last line. 
Furthermore, for $1\leq m\leq k$, we get that
\begin{equation}
   \abs{\frac{d^m}{dx^m} \sum_{v\in \mathcal{V}_1}\Phi^{N,1}_{j+v}(x)} \leq \alpha^m \sigma^{(m)}\left(\frac{\alpha}{N}\right) \leq \epsilon,
\end{equation}
where we used \eqref{eq:V1} and the monotonic decay of $\sigma^{(m)}(x)$ for $x\in [\alpha/N,\infty)$ and our choice of $\alpha$ (cf. equation \eqref{eq:alpha}).
This allows us to conclude that
\begin{equation}
    \norm{\sum_{v\in \mathcal{V}_1}\Phi^{N,1}_{j+v}(x)-1}_{W^{k,\infty}(I_j^N)} \leq \epsilon. 
\end{equation}
For the induction step, we assume that for some $2 \leq D \leq d$ it holds that
\begin{equation}
    \norm{\sum_{v\in \mathcal{V}_{D-1}}\Phi^{N,D-1}_{j+v}-1}_{W^{k,\infty}(I_j^N)} \leq 2^{(D-1)k}(D-1) \epsilon.
\end{equation}  
Using Lemma \ref{lem:leibniz}, we find that for $x\in I_j^N$,
\begin{align}
\begin{split}
    \norm{\sum_{v\in \mathcal{V}_D}\Phi^{N,D}_{j+v}(x)-1}_{W^{k,\infty}(I_j^N)} &= \norm{\sum_{w\in \mathcal{V}_1}  \rho_{j_D+w}^{N}(x_D) \sum_{v\in \mathcal{V}_{D-1}}\Phi^{N,D-1}_{j+v}(x)-1}_{W^{k,\infty}(I_j^N)}\\
    &\leq \norm{\sum_{w\in \mathcal{V}_1}  \rho_{j_D+w}^{N}(x_D) -1}_{W^{k,\infty}(I_j^N)} \\
    &\quad + 2^k\norm{\sum_{w\in \mathcal{V}_1}  \rho_{j_D+w}^{N}(x_D)}_{W^{k,\infty}(I_j^N)} \norm{\sum_{v\in \mathcal{V}_{D-1}}\Phi^{N,D-1}_{j+v}(x)-1}_{W^{k,\infty}(I_j^N)}\\
    &\leq \epsilon + 2^k  2^{(D-1)k}(D-1)\epsilon \leq 2^{Dk} D\epsilon.
\end{split}    
\end{align}
This concludes the proof. 
\end{proof}

\begin{lemma}\label{lem:indicator-far}
Let $k\in\mathbb{N}_0$ and $v\in\mathbb{Z}^d$ with $\norm{v}_\infty\geq 2$. Then it holds that
\begin{equation}
     \norm{\Phi^{N,d}_{j+v}}_{W^{k,\infty}(I_j^N)} \leq \max\{1,(2k)^{2k} \alpha^k\} \epsilon.
\end{equation}
\end{lemma}
\begin{proof}
Let $x\in I_j^N$ and let $1\leq i \leq d$ be an index such that $\abs{v_i}\geq 2$. Using some basic equalities for the hyperbolic tangent function and the definition of $\alpha$, we obtain that
\begin{align}
    \begin{split}
        \abs{\rho_{j_i+v_i}^N(x_i)} &\leq \frac{1}{2}\sigma\left(\frac{2\alpha}{N}\right)-\frac{1}{2}\sigma\left(\frac{\alpha}{N}\right)\\
        &= \frac{1}{2}\sigma\left(\frac{\alpha}{N}\right) \left(1-\sigma\left(\frac{2\alpha}{N}\right)\sigma\left(\frac{\alpha}{N}\right)\right) \\
        & \leq  \frac{1}{2} \left(1-\sigma^2\left(\frac{\alpha}{N}\right)\right) \leq \epsilon. 
    \end{split}
 \end{align}
In addition, for every $1\leq \ell \leq d$, it holds that $\abs{\rho_{j_\ell+v_\ell}^N(x_\ell)} \leq 1$. This implies that
\begin{equation} \label{eq:42-linf}
     \norm{\Phi^{N,d}_{j+v}}_{L^\infty(I_j^N)} \leq \epsilon.
\end{equation}
Let $1\leq m\leq k$, then it holds that (by our choice of the index $i$),
\begin{equation}
    \abs{\frac{d^m}{dx_i^m} \rho_{j_i+v_i}^N(x_i)} \leq \alpha^m \abs{\sigma^{(m)}\left(\frac{\alpha}{N}\right)} \leq \epsilon. 
\end{equation}
Now let $\beta\in\mathbb{N}^d$ such that $1\leq \abs{\beta}\leq k$. Then 
\begin{equation} \label{eq:42-wk}
    \abs{D^\beta \Phi^{N,d}_{j+v}(x)} = \abs{\prod^{d}_{\ell=1} \frac{d^{\beta_\ell}}{dx_\ell^{\beta_\ell}}\rho_{j_\ell+v_\ell}^{N_\ell}(x_\ell)} \leq \epsilon \prod^{d}_{\ell=1, \ell\neq i, \beta_\ell \neq 0} (2\beta_\ell)^{\beta_\ell+1} \alpha^{\beta_\ell} \leq \epsilon (2k)^{2k} \alpha^k,
\end{equation}
where we used the fact that $\abs{\sigma^{(m)}(x)} \le (2m)^{m+1}$ in the first inequality (cf. Lemma \ref{lem:bound-der-tanh}). Combining \eqref{eq:42-linf} and \eqref{eq:42-wk} proves the statement. 
\end{proof}

\section{Main results}
\label{sec:5}

\subsection{Approximation of functions in Sobolev spaces}

We now present the first main result of the paper. It follows from the lemma of Bramble--Hilbert (Lemma \ref{lem:bramble-hilbert2}) that localized Taylor polynomials can approximate a function $f\in W^{s,\infty}([0,1]^d)$. For functions $f\in C^{s}([0,1]^d)$, this approximation follows from Taylor's theorem (Lemma \ref{lem:taylor}). We then use the results from the previous two sections to construct tanh neural networks that approximate localized Taylor polynomials in Sobolev norm. We prove that the function $f$ can be approximated by a tanh neural network with two hidden layers and we provide explicit bounds on the width and approximation error.

\begin{theorem}\label{thm:main}
Let $d,s \in \mathbb{N}$, $R>0$ as in \eqref{eq:alpha}, $\delta>0$ and $f\in W^{s,\infty}([0,1]^d)$. There exist constants $\mathcal{C}(d,k,s,f)$, $N_0(d)>0$, such that for every $N\in\mathbb{N}$ with $N>N_0(d)$ there exists a tanh neural network $\widehat{f}^N$ with two hidden layers, one of width at most $3\left\lceil\frac{s}{2}\right\rceil\abs{P_{s-1,d+1}}+d(N-1)$ and another of width at most $3\left\lceil\frac{d+2}{2}\right\rceil\abs{P_{d+1,d+1}}N^d$ (or $3\left\lceil\frac{s}{2}\right\rceil+N-1$ and $6N$ for $d=1$), such that,
\begin{equation}
     \norm{f-\widehat{f}^N}_{L^\infty([0,1]^d)} \leq  \left(1+\delta\right) \frac{\mathcal{C}(d,0,s,f)}{N^{s}}, 
\end{equation}
and {\color{black}for $k=1, \ldots, s-1$},
\begin{equation}
    \ckunit{f-\widehat{f}^N} \leq  3^d \left(1+{\color{black}\delta}\right) (2(k+1))^{3k} \max\left\{R^k,\ln^k\left(\beta N^{s+d+2}\right)\right\}\frac{\mathcal{C}(d,k,s,f)}{N^{s-k}}, 
\end{equation}
where we define 
\begin{equation}
    \beta = \frac{k^3 2^{d}\sqrt{d}\max\{1,\ckunit{f}^{1/2}\}}{\delta\min\{1,\sqrt{\mathcal{C}(d,k,s,f)}\} }. 
\end{equation}
If $f\in C^{s}([0,1]^d)$, then it holds that
\begin{equation}
    \mathcal{C}(d,k,s,f) = \max_{0\leq \ell \leq k} \frac{1}{(s-\ell)!}\left(\frac{3d}{2}\right)^{s-\ell}\abs{f}_{W^{s,\infty}([0,1]^d)}, \quad N_0(d) = \frac{3d}{2},
\end{equation}
and else it holds that
\begin{equation}
    \mathcal{C}(d,k,s,f) = \max_{0\leq \ell \leq k} \frac{\pi^{1/4}\sqrt{s}}{(s-\ell-1)!} \left(5d^2\right)^{s-\ell}\abs{f}_{W^{s,\infty}([0,1]^d)}, \quad N_0(d) = 5d^2. 
\end{equation}
In addition, the weights of $\widehat{f}^N$ scale as $O\left({\color{black}\mathcal{C}}^{-s/2}N^{d(d+s^2+k^2)/2}(s(s+2))^{3s(s+2)}\right)$.
\end{theorem}
\begin{proof}
We will prove the theorem in the following manner. We divide the unit cube into $N^d$ cubes of edge length $1/N$. On each of these cubes, $f$ can be  approximated in Sobolev norm by a polynomial. The global approximation can then be constructed by multiplying each polynomial with the indicator function of the corresponding cubes and summing over all cubes. We then prove that replacing these polynomials, multiplications and indicator functions with the tanh neural networks from the previous sections results in a new approximation that has approximately the same accuracy. In the last step we will calculate the size of the required neural network. 

\textbf{Step 1: construction of the approximation. }
Let us denote $J_j^N = \bigtimes_{i=1}^d \left((j_i-2)/N,(j_i+1)/N\right)$. We calculate that $\text{diam}(J_j^N) = \frac{3\sqrt{d}}{N}$ and that there exists a ball with diameter $\frac{1}{\sqrt{d}} \text{diam}(J_j^N)$ such that $J_j^N$ is star-shaped with respect to every point in this ball. As a consequence, the Bramble-Hilbert lemma (Lemma \ref{lem:bramble-hilbert2}) ensures the existence of a polynomial $p_j^N$ of degree at most $s-1$ such that
\begin{align}\label{eq:pkn-acc-sobolev}
\begin{split}
    \norm{f-p_j^N}_{W^{\ell,\infty}(J_j^N)} &\leq \frac{\pi^{1/4}\sqrt{s}}{(s-\ell-1)!} \left(\frac{5d^2}{N}\right)^{s-\ell} \abs{f}_{W^{s,\infty}([0,1]^d)}  \\ 
    &\leq \max_{0\leq m\leq \ell} \frac{\pi^{1/4}\sqrt{s}(5d^2)^{s-m}}{(s-m-1)!}\frac{\abs{f}_{W^{s,\infty}([0,1]^d)}}{N^{s-\ell}} =: \frac{\mathcal{C}(d,\ell,s,f)}{N^{s-\ell}},
\end{split}
\end{align}
for all $0\leq \ell \leq k$, under the assumption that $N>5d^2$, and where we used that $3\sqrt{e}\leq 5$.
If moreover $f\in C^{s}([0,1]^d)$, then Taylor's theorem (Lemma \ref{lem:taylor} with $\delta = \frac{3}{2N}$) ensures the existence of a polynomial $p_j^N$ of degree at most $s-1$ such that
\begin{align}
\begin{split}
    \norm{f-p_j^N}_{W^{\ell,\infty}(J_j^N)} &\leq \frac{1}{(s-\ell)!}\left(\frac{3d}{2N}\right)^{s-\ell}\abs{f}_{W^{s,\infty}([0,1]^d)} \\
    &\leq \max_{0\leq m\leq \ell} \frac{1}{(s-m)!}\left(\frac{3d}{2}\right)^{s-m} \frac{\abs{f}_{W^{s,\infty}([0,1]^d)}}{N^{s-\ell}}
    =: \frac{\mathcal{C}(d,\ell,s,f)}{N^{s-\ell}},
\end{split}
\end{align}
for all $0\leq \ell \leq k$, under the assumption that $N>3d/2$. The remainder of the argument will be independent of which polynomial $p_j^N$ and which definition of $\mathcal{C}(d,\ell,s,f)$ is used. To simplify notation, we also define $p^N=\sum_j p_j^N \chi_j$, where $\chi_j$ denotes the indicator function on $I_j^N$. 
Next, let $q_j^N$ be a tanh neural network as in Section 1 such that 
\begin{equation}\label{eq:qkn-acc-sobolev}
    \norm{q_j^N-p_j^N}_{W^{k,\infty}([0,1]^d)} \leq \eta. 
\end{equation}
In addition, we define 
\begin{equation}
    q_j^N(x)\widehat{\times} \Phi^{N,d}_j(x) := \widehat{\times}_{d+1}^{h}(q_j^N(x),\phi_{j_1}^{N,d}(x_1), \ldots, \phi_{j_d}^{N,d}(x_d)),
\end{equation}
where $\widehat{\times}_{d+1}^{h}$ is the network from Corollary \ref{cor:mult-shallow} and $h=h(N)$ will be defined in the remainder of the proof. 
We then define our approximation as
\begin{equation}
    \widehat{f}^N(x) = \sum_{j\in \{1,\ldots, N\}^d} q_j^N(x)\widehat{\times} \Phi^{N,d}_j(x). 
\end{equation}

\textbf{Step 2: estimating the error of the approximation. }The triangle inequality gives us
\begin{align}\label{eq:three-terms-sobolev}
\begin{split}
    \ckunit{f-\widehat{f}^N} &\leq \ckunit{f-\sum_{j\in \{1,\ldots, N\}^d} f \cdot \Phi^{N,d}_j} +  \ckunit{\sum_{j\in \{1,\ldots, N\}^d} (f -q_j^N)\cdot \Phi^{N,d}_j}\\
    &+ \ckunit{\sum_{j\in \{1,\ldots, N\}^d} (q_j^N \cdot \Phi^{N,d}_j-q_j^N \widehat{\times} \Phi^{N,d}_j)}
    \end{split}
\end{align}
We proceed by bounding each term of the right hand side separately. 

\textit{Step 2a: First term of \eqref{eq:three-terms-sobolev}.} Let $i\in\{0,\ldots, N\}^d$ be arbitrary. 
Recalling that $\mathcal{V}_d = \{v\in\mathbb{Z}^d: \norm{v}_\infty\leq 1\}$, we observe for $k\geq 1$,
\begin{align}
    \begin{split}
        \cki{f-\sum_{j\in \{1,\ldots, N\}^d} f \cdot \Phi^{N,d}_j} &\leq 2^k \cki{f}\cki{1-\sum_{v\in \mathcal{V}_d} \Phi^{N,d}_{i+v}} \\ & \quad +2^k \cki{f}\cki{ \sum_{\substack{j\in \{1,\ldots, N\}^d \\ j-i\not\in\mathcal{V}_d}} \Phi^{N,d}_{j}}\\
        &\leq 2^k \cki{f} (2^{kd}d\epsilon + N^d (2k)^{2k}\alpha^k \epsilon)\\
        &\leq 2^k \cki{f} 2^{kd}d\epsilon \\
        & \quad + 2^k \cki{f} N^d (2k)^{2k}N^k (k+1)^k \max\left\{R^k,\ln^k\left(\frac{2Nk^2}{\epsilon^{\frac{1}{k+1}}e}\right)\right\}\epsilon\\
        & \leq \delta (2(k+1))^{3k} \max\left\{R^k,\ln^k\left(\frac{2Nk^2}{\epsilon^{\frac{1}{k+1}}e}\right)\right\}\frac{\mathcal{C}(d,k,s,f)}{N^{s-k}},  
    \end{split}
\end{align}
where we used Lemma \ref{lem:leibniz}, Lemma \ref{lem:indicator-close}, Lemma \ref{lem:indicator-far} and Lemma \ref{lem:alpha-growth}, as well as a {\color{black}suitable definition of $\epsilon$, satisfying
\begin{align}\label{eqn:def-eps-1}
\epsilon 
\le
\frac{
\delta \mathcal{C}(d,k,s,f)
}{
2^{(k+1)d}dN^{s+d}\ckunit{f}
}.
\end{align}
} 
Analogously, for $k=0$, one can obtain that 
\begin{equation}
    \linfi{f-\sum_{j\in \{1,\ldots, N\}^d} f \cdot \Phi^{N,d}_j} \leq \linfi{f}( d\epsilon + N^d \epsilon) \leq \frac{\delta}{3} \frac{\mathcal{C}(d,k,s,f)}{N^{s-k}}.
\end{equation}

\textit{Step 2b: Second term of \eqref{eq:three-terms-sobolev} for $k=0$.} In order to bound the second term, we first make some auxiliary calculations. To begin with, we consider the case where $k=0$. We find that
\begin{align}
\begin{split}
\left|
\sum_{v\in \mathcal{V}_d}
(f-q^N_{i+v}) \Phi^{N,d}_{i+v}
\right|
&\le
\max_{v\in \mathcal{V}_d} |f-q^N_{i+v}|
\left(
\sum_{v\in \mathcal{V}_d} |\Phi^{N,d}_{i+v}|
\right)
\\
&=
\max_{v\in \mathcal{V}_d} |f-q^N_{i+v}|
\left|\sum_{v\in \mathcal{V}_d} \Phi^{N,d}_{i+v}\right|
\end{split}
\end{align}
where all functions are evaluated at some $x\in I_i^N$. We can then use the bounds
\begin{equation}
    \abs{f-q_{i+v}^N} \leq \frac{\mathcal{C}(d,0,s,f)}{N^{s}} + \eta,
\end{equation}
which follows from \eqref{eq:pkn-acc-sobolev} and \eqref{eq:qkn-acc-sobolev}, and, 
\begin{equation}
   \abs{\sum_{v\in \mathcal{V}_d} \Phi^{N,d}_{i+v}}  \leq 1+d\epsilon,
\end{equation}
which follows from Lemma \ref{lem:indicator-close}.
As a consequence, we find that
\begin{align}
\begin{split}
    \norm{\sum_{v\in \mathcal{V}_d} (f-q_{i+v}^N) \Phi^{N,d}_{i+v}}_{L^\infty(I_i^N)} &\leq   \left(\frac{\mathcal{C}(d,k,s,f)}{N^{s-k}} + \eta\right)(1+d\epsilon).
\end{split}
\end{align}
Combining this result with the triangle inequality,  \eqref{eq:pkn-acc-sobolev}, \eqref{eq:qkn-acc-sobolev} and Lemma \ref{lem:indicator-far}, we find that
\begin{align}
    \begin{split}
       &\linfi{\sum_{j\in \{1,\ldots, N\}^d} (f -q_j^N)\cdot \Phi^{N,d}_j}\\ &\leq \linfi{\sum_{v\in \mathcal{V}_d} (f-q_{i+v}^N) \Phi^{N,d}_{i+v}} \quad+ \sum_{\substack{j\in \{1,\ldots, N\}^d \\ j-i\not\in\mathcal{V}_d}} \linfi{(f-q_{j}^N)} \linfi{\Phi^{N,d}_{j}}\\
       &\leq   \left(\frac{\mathcal{C}(d,k,s,f)}{N^{s-k}} + \eta\right)(1+d\epsilon) + N^d \left(\mathcal{C}(d,k,s,f)+\eta \right) \epsilon\\
        &\leq   \left(1 +\frac{\delta}{3}\right)\frac{\mathcal{C}(d,k,s,f)}{N^{s-k}}.
    \end{split}
\end{align}
where we obtain the last inequality by {\color{black}making a suitable choice of $\eta$ and $\epsilon$.} 

\textit{Step 2c: Second term of \eqref{eq:three-terms-sobolev} for $k\geq1$.} Next we consider the case where $0<k<s$. Let $\beta\in\mathbb{N}_0^d$ be such that $\abs{\beta}\leq k$. Then as a consequence of the general Leibniz rule we find that
\begin{equation}
    D^\beta \left(\sum_{v\in \mathcal{V}_d} (f-q_{i+v}^N) \Phi^{N,d}_{i+v}\right) \leq \sum_{\beta'\leq \beta}\binom{\beta}{\beta'} \sum_{v\in \mathcal{V}_d} \abs{D^{\beta'}(f-q_{i+v}^N)}\abs{D^{\beta-\beta'} \Phi^{N,d}_{i+v}}
\end{equation}
where all functions are evaluated at some $x\in I_i^N$. For every $v\in \mathcal{V}_d$ and $\beta'\leq \beta$ with $\ell := \abs{\beta-\beta'}$, we can then use the bounds
\begin{equation}
    \abs{D^{\beta'}(f-q_{i+v}^N)} \leq \norm{f-q_{i+v}^N}_{W^{k-\ell,\infty}(I^N_i)} \leq \frac{\mathcal{C}(d,k-\ell,s,f)}{N^{s-k+\ell}} + \eta,
\end{equation}
which follows from \eqref{eq:pkn-acc-sobolev} and \eqref{eq:qkn-acc-sobolev}, and, 
\begin{equation}
   \abs{D^{\beta-\beta'} \Phi^{N,d}_{i+v}} \leq \alpha^\ell (2\ell)^{2\ell} = N^\ell (2\ell)^{2\ell} \max\left\{R^\ell,\ln^\ell\left(\frac{(2k)^{k+1}(Nk)^k}{e^k\epsilon}\right)\right\},
\end{equation}
which follows from Lemma \ref{lem:bound-der-tanh} and Lemma \ref{lem:alpha-growth}.
As $\sum_{\beta'\leq \beta}\binom{\beta}{\beta'} \leq 2^k$ (as a consequence of the multi-binomial theorem), we find that
\begin{align}
\begin{split}
    \cki{\sum_{v\in \mathcal{V}_d} (f-q_{i+v}^N) \Phi^{N,d}_{i+v}} &\leq 2^k 3^d \left(\frac{\mathcal{C}(d,k,s,f)}{N^{s-k}} + \eta N^k\right)(2k)^{2k} \max\left\{R^k,\ln^k\left(\frac{(2Nk^2)^{k+1}}{\epsilon e^k}\right)\right\}.
\end{split}
\end{align}
Combining this result with the triangle inequality, Lemma \ref{lem:alpha-growth}, Lemma \ref{lem:leibniz}, \eqref{eq:pkn-acc-sobolev}, \eqref{eq:qkn-acc-sobolev}, Lemma \ref{lem:indicator-far} and the fact that $\ln(x)\leq \sqrt{x}$ for $x>0$, we find that
\begin{align}
    \begin{split}
       &\cki{\sum_{j\in \{1,\ldots, N\}^d} (f -q_j^N)\cdot \Phi^{N,d}_j}\\ &\leq \cki{\sum_{v\in \mathcal{V}_d} (f-q_{i+v}^N) \Phi^{N,d}_{i+v}} \quad+ \sum_{\substack{j\in \{1,\ldots, N\}^d \\ j-i\not\in\mathcal{V}_d}} \cki{(f-q_{j}^N) \Phi^{N,d}_{j}}\\
       &\leq \cki{\sum_{v\in \mathcal{V}_d} (f-q_{i+v}^N) \Phi^{N,d}_{i+v}} \quad+ \sum_{\substack{j\in \{1,\ldots, N\}^d \\ j-i\not\in\mathcal{V}_d}} 2^k \cki{(f-q_{j}^N)} \cki{\Phi^{N,d}_{j}}\\
       &\leq 2^k 3^d \left(\frac{\mathcal{C}(d,k,s,f)}{N^{s-k}} + \eta N^k\right)(2k)^{2k} \max\left\{R^k,\ln^k\left(\frac{(2Nk^2)^{k+1}}{\epsilon e^k}\right)\right\} \\ &\quad  + N^d 2^k\left(\mathcal{C}(d,k,s,f)+\eta \right)(2k)^{2k} N^k (k+1)^k \left(\frac{2Nk^2}{e}\right)^{k/2}\sqrt{\epsilon} \\
       &\leq 3^d \left(1+\frac{\delta}{3}\right) (2(k+1))^{3k} \max\left\{R^k,\ln^k\left(\frac{2Nk^2}{\epsilon^{\frac{1}{k+1}}e}\right)\right\}\frac{\mathcal{C}(d,k,s,f)}{N^{s-k}},
    \end{split}
\end{align}
where we obtain the last inequality by {\color{black}making a suitable choice of $\eta$ and $\epsilon$, satisfying
\begin{equation}\label{eqn:def-eps-2}
    \eta \leq \frac{\delta \mathcal{C}}{6N^{s}} \quad \text{and} \quad \epsilon \leq \frac{\delta^2 }{N^{2s+2d+k}k^k}, 
\end{equation}
where we assumed that $0< \delta < 5/6$.
} 

\textit{Step 2d: Third term of \eqref{eq:three-terms-sobolev}.} Finally, using the triangle inequality, Lemma \ref{lem:faa-di-bruno}, Corollary \ref{cor:mon-mult1} and Lemma \ref{lem:bound-der-tanh} we obtain that for some $C_k>0$ depending only on $k$, 
\begin{align}
    \begin{split}
&\cki{\sum_{j\in \{1,\ldots, N\}^d} (q_j^N \cdot \Phi^{N,d}_j-q_j^N \widehat{\times} \Phi^{N,d}_j)} \\&\leq N^d C_k (d+1)^d d^{2k}\cdot \ck{\widehat{\times}_{d+1}^{h} 
{\color{black}- \prod_{i=1}^{d+1} x_i}
}\left(\ckunit{f}+\ckunit{f-q_j^N} + \ckunit{\rho^N_i}\right)^k\\
&\leq N^d C_k (d+1)^d d^{2k}\cdot h \left(\ckunit{f}+\frac{\mathcal{C}(d,k,s,f)}{N^{s-k}} + \eta + (2\alpha k)^{k+1}\right)^k\\
&\leq 2^k \frac{\delta}{3} \frac{\mathcal{C}(d,k,s,f)}{N^{s-k}},
    \end{split}
\end{align}
where we obtain the last inequality by {\color{black}making a suitable choice of $h$ and $\eta$, satisfying}
\begin{equation}\label{eqn:def-h}
    \eta \leq \ckunit{f} \quad \text{and} \quad h \leq \frac{2^k \delta \mathcal{C}}{3N^{d+s-k}C_k (d+1)^d d^{2k}(2\ckunit{f}+ \mathcal{C}+ (2\alpha k)^{k+1})^k}.
\end{equation}

\textit{Step 2e: Final error bound.} As $i$ was chosen arbitrary, combining the contributions from the three terms of \eqref{eq:three-terms-sobolev} then proves that
\begin{equation}\label{eqn:def-eps-3}
     \norm{f-\widehat{f}^N}_{L^\infty([0,1]^d)} \leq  \left(1+\delta\right) \frac{\mathcal{C}(d,0,s,f)}{N^{s}}. 
\end{equation}
Moreover, from \eqref{eqn:def-eps-1} and \eqref{eqn:def-eps-2} we find that a suitable definition of $\epsilon$ is given by
\begin{equation}
    \epsilon = \frac{\delta^2\min\{1,\mathcal{C}(d,k,s,f)\} }{N^{2s+2d+k}k^k2^{(k+1)d}d\max\{1,\ckunit{f}\}},
\end{equation}
from which it follows that for $k\geq 1$,
\begin{equation}
    \epsilon^{\frac{1}{k+1}} \geq \frac{\delta\min\{1,\sqrt{\mathcal{C}(d,k,s,f)}\} }{N^{s+d+1}k2^{d}\sqrt{d}\max\{1,\ckunit{f}^{1/2}\}}.
\end{equation}
Combining this observation with all previous steps of the proof then leads to the error bound
\begin{equation}
    \ckunit{f-\widehat{f}^N} \leq  3^d \left(1+{\color{black}\delta}\right) (2(k+1))^{3k} \max\left\{R^k,\ln^k\left(\beta N^{s+d+2}\right)\right\}\frac{\mathcal{C}(d,k,s,f)}{N^{s-k}}, 
\end{equation}
for $k\geq 1$ and where we define 
\begin{equation}
    \beta = \frac{k^3 2^{d}\sqrt{d}\max\{1,\ckunit{f}^{1/2}\}}{\delta\min\{1,\sqrt{\mathcal{C}(d,k,s,f)}\} }. 
\end{equation}

\textbf{Step 3: Estimating the network and weights sizes. } The first hidden layer requires $3\left\lceil\frac{s}{2}\right\rceil\abs{P_{s-1,d+1}}$ neurons for the computation of all multivariate monomials (cf. Corollary \ref{cor:mon-mult1}). For $d=1$, the result follows from Lemma \ref{lem:monomials} instead of Corollary \ref{cor:mon-mult1}. For the computation of all $\rho^{N}_j(x_i)$ another $d(N-1)$ neurons are needed in the first hidden layer. The second hidden layer needs at most $3\left\lceil\frac{d+2}{2}\right\rceil\abs{P_{d+1,d+1}}$ neurons for realizing $\widehat{\times}_{d+1}^{h}$, which needs to be performed $N^d$ times. For $d=1$, six neurons are sufficient to approximate the multiplication (see \eqref{eq:xy}). 

In the proof we achieved the wanted accuracy by making suitable choices of $\eta, \epsilon, h$.
From equation \eqref{eqn:def-eps-3} and Lemma \ref{lem:alpha-growth}, it follows that 
\begin{equation}
    \alpha = O\left(Ns \ln(\mathcal{C}N)\right).
\end{equation}

For the approximate multiplication, \eqref{eqn:def-h} requires that $h^{-1}=O(N^{d+s+k^2}s^{k^2+k})$. Corollary \ref{cor:mult-shallow} then proves that the weights of $\widehat{\times}^h_{d+1}$ grow as $O(N^{d(d+s+k^2)/2}s^{d(k^2+k)/2})$. Finally, the condition $\eta^{-1} = O(\mathcal{C}^{-1}N^{s})$ from \eqref{eqn:def-eps-2} corresponds to weights growing as
\begin{equation}
O\left(\mathcal{C}^{-s/2}N^{s^2/2}(s(s+2))^{3s(s+2)}\right)
\end{equation}
as a consequence of Corollary \ref{cor:mon-mult1}. Calculating the maximum size of all the weights concludes the proof.
\end{proof}

\begin{remark}
The result of Theorem \ref{thm:main} can be generalized to functions $f\in W^{k,p}(\Omega)$ for $p<\infty$. For this, a slightly more general version of Lemmas \ref{lem:leibniz} and \ref{lem:bramble-hilbert2} is needed. The convergence rate will still be as in Theorem \ref{thm:main}, only the constant $\mathcal{C}$ will be different. 
\end{remark}

\begin{remark}
Recently, it has been shown that the curse of dimensionality can be lessened for functions in so-called Korobov spaces \cite{montanelli2019new,li2019better,blanchard2020representation}. In particular, in \cite[Theorem 4.2]{li2019better}, this framework is used to show how ReQU neural networks can approximate a $C^k([-1,1]^d)$-function to an accuracy of $\epsilon>0$ in supremum norm with at most  $O\big(\frac{d}{k}\ln(\frac{1}{\epsilon})+d\big)$ hidden layers and at most $ O\big(\epsilon^{-\frac{1+\delta}{k}}\big(\frac{1+\delta}{k}\ln(\frac{1}{\epsilon})\big)^{d-1}\big)$ neurons and non-zero weights. As their proof builds upon the mimicking of polynomials, it is clear from our results that similar approximation rates can be obtained using tanh neural networks.
\end{remark}

One particularly useful consequence of Theorem \ref{thm:main} is that it provides an explicit error bound on the approximation of Lipschitz functions using tanh neural networks. 
\begin{corollary}
Let $d\in\mathbb{N}$ and let $f:[0,1]^d\to\mathbb{R}$ be a Lipschitz continuous function with Lipschitz constant $L>0$. For every $N\in \N$ with $N>5d^2$ there exists a tanh neural network $\widehat{f}^N$ with two hidden layers of widths at most $d(N-1)$ and $3\left\lceil\frac{d+1}{2}\right\rceil\abs{P_{d,d}}N^d$ (or $N-1$ and $6N$ for $d=1$), such that
\begin{equation}
     \norm{f-\widehat{f}^N}_{L^\infty([0,1]^d)} \leq   \frac{7d^2L}{N}.
\end{equation}
\end{corollary}
\begin{proof}
The corollary follows directly from Theorem \ref{thm:main} by setting $k=0$, $s=1$, choosing $\delta>0$ in such a way that $5(1+\delta)\pi^{1/4}\leq 7$ and observing that $\abs{f}_{W^{1,\infty}}\leq L$ because of the Lipschitz continuity of $f$. The constructed network in Theorem \ref{thm:main} is based on localized $(s-1)$-th order polynomials. For $s=1$ this corresponds to constant functions, thereby removing the need to mimick monomials. As a consequence, the network width can be simplified to the widths stated in the corollary. 
\end{proof}

\subsection{Approximation of analytic functions}
We now investigate how we can apply Theorem \ref{thm:main} to analytic functions. As the class of analytic functions coincides with the Gevrey class $G^1$, it follows that for every analytic function there exists a constant $C_f>0$ such that
\begin{equation}
    \abs{f}_{W^{s,\infty}([0,1]^d)}\leq C_f^{s+1} s! \text{ for all } s\in\mathbb{N}_0. 
\end{equation}
A related concept is the class of $(Q,R)$-analytic functions \cite{demanet2010chebyshev,candes2007fast,candes2009fast}, where $Q,R>0$, consisting of analytic functions for which the following smoothness condition holds, 
\begin{equation}
    \abs{f}_{W^{s,\infty}([0,1]^d)}\leq QR^{-s} s! \text{ for all } s\in\mathbb{N}_0. 
\end{equation}
Note that any analytic function is $(C_f, C_f^{-1})$-analytic by the previous characterization of analyticity. Hence, a function is analytic on some compact interval if and only if it is $(Q,R)$-analytic for some $Q,R>0$ on that interval. The following corollaries discuss multiple ways to approximate analytic functions using tanh neural networks. 
All constants mentioned in the statements can be easily deduced from the proofs. 

We start with the basic consequence of Theorem \ref{thm:main} for $(Q,R)$-analytic functions. It provides explicit estimates on both the approximation error in supremum norm and the network size. It can easily be generalized to Sobolev norm using Theorem \ref{thm:main}. 

\begin{corollary}\label{cor:analytic1}
Let $d\in \mathbb{N}$, $\delta,Q,R>0$, $\Omega\subset \mathbb{R}^d$ open with $[0,1]^d \subset \Omega$ and let f be $(Q,R)$-analytic on $\Omega$. Then for every $s\in\mathbb{N}_0,N\in\mathbb{N}$ with $N>3d/2$, there is a tanh neural network $\widehat{f}^{N,s}$ with two hidden layers of widths at most $3\left\lceil\frac{s}{2}\right\rceil\abs{P_{s-1,d+1}}+d(N-1)$ and $3\left\lceil\frac{d+2}{2}\right\rceil\abs{P_{d+1,d+1}}N^d$ (or $3\left\lceil\frac{s}{2}\right\rceil+N-1$ and $6N^d$ for $d=1$) such that
   \begin{equation}
     \linfunit{f-\widehat{f}^{N,s}} \leq (1+\delta) Q  \left(\frac{3d}{2RN}\right)^{s}.
\end{equation} 
Moreover, if $R>d/2$ then for every $s\in\mathbb{N}_0$ there is a shallow tanh neural network $\widehat{f}^{s}$ with at most $\frac{3s}{2}\abs{P_{s-1,d+1}}$ (or $3\left\lceil\frac{s}{2}\right\rceil$ for d=1) in its hidden layer such that
   \begin{equation}
     \linfunit{f-\widehat{f}^{s}} \leq (1+\delta) Q \left(\frac{d}{2R}\right)^{s}.
\end{equation} 
\end{corollary}
\begin{proof}
The first part of the statement follows directly from Theorem \ref{thm:main} by taking $\delta=1/3$. The second part follows from taking $N=1$ in Theorem \ref{thm:main} and observing that the proof can be simplified in this case. Indeed, one can then directly use Taylor's theorem (Lemma \ref{lem:taylor}) with $\delta = \frac{1}{2}$ instead of $\delta = \frac{3}{2N}$ and there is no more need for an approximate partition of unity, thereby also removing the need for a second hidden layer.  
\end{proof}

The following corollary enables a consistent comparison with the available literature, as it bounds the approximation error in terms of one single parameter. Whereas other papers focus on the number of non-zero weights and biases as complexity measure, we opted for the network width. This is useful in practice as the network width can be directly chosen, whereas it is very challenging to exactly control the sparsity of the neural network (i.e. the number of non-zero weights and biases). Moreover, many bounds on the generalization error require an estimate of the network width \cite{beck2019full,jentzen2020overall}. 

\begin{corollary}\label{cor:analytic2}
Let $d\in \mathbb{N}$, $k\in \mathbb{N}_0$, $\delta, Q,R>0$, $\Omega\subset \mathbb{R}^d$ open with $[0,1]^d \subset \Omega$ and let f be $(Q,R)$-analytic on $\Omega$. Then there exists a constant $c_{d,k,\alpha,f}>0$ such that for every $\mathcal{N}\in\mathbb{N}$ there exists a tanh neural network $\widehat{f}^\mathcal{N}$ with two hidden layers of width at most $O(\mathcal{N})$ for $\mathcal{N}\to \infty$ such that
\begin{equation}
     \ckunit{f-\widehat{f}^\mathcal{N}} \leq c_{d,k,\alpha,f} \mathcal{N}^{\frac{k}{d+1}}\exp\left(-\alpha\mathcal{N}^{\frac{1}{d+1}} \log(\mathcal{N})\right) \leq \frac{c_{d,k,\alpha,f}}{\mathcal{N}^{\alpha-k/(d+1)}}.
\end{equation}
In particular, for $k=0$ it holds that
\begin{equation}
     \norm{f-\widehat{f}^\mathcal{N}}_{L^{\infty}([0,1]^d)} \leq  (1+\delta)Q\cdot  \exp\left(-\alpha\mathcal{N}^{\frac{1}{d+1}} \log(\mathcal{N})\right) \leq \frac{(1+\delta)Q}{\mathcal{N}^{\alpha}}.
\end{equation}
\end{corollary}
\begin{proof}
First we observe that for every $\gamma>0$ it holds that
\begin{equation}
    \ln^k\left(\beta N^{s+d+2}\right) = O\left(\left(\frac{2NR}{3d}\right)^{\gamma s}\right)
\end{equation}
for large $s$ and $N$. From Theorem \ref{thm:main} we then find that
for every $N$ and $s$ there is a network $\widehat{f}^{N,s}$ such that
\begin{equation}
     \norm{f-\widehat{f}^{N,s}}_{W^{k,\infty}([0,1]^d)} = O\left( \left(\frac{s}{R}\right)^k\left(\frac{3d}{2NR}\right)^{s(1-\gamma)-k}\right)
\end{equation}
From Theorem \ref{thm:main} with the choices $s=k+\alpha(1-\gamma)^{-1} (d+1)\mathcal{N}^{\frac{1}{d+1}}$ and $N=\frac{3d}{2R}\mathcal{N}^{\frac{1}{d+1}}$ for some $\mathcal{N}\in\mathbb{N}$, gives that there exists a constant $c_{d,k,\alpha,f}>0$ such that
\begin{align}
\begin{split}
    \norm{f-\widehat{f}^\mathcal{N}}_{W^{k,\infty}([0,1]^d)} &\leq c_{d,k,\alpha,f} \mathcal{N}^{\frac{k}{d+1}} \left(\frac{1}{\mathcal{N}^{\frac{1}{d+1}}}\right)^{\alpha (d+1)\mathcal{N}^{\frac{1}{d+1}}} \\& =  c_{d,k,\alpha,f} \mathcal{N}^{\frac{k}{d+1}}\exp\left(-\alpha\mathcal{N}^{\frac{1}{d+1}} \log(\mathcal{N})\right) \\&\leq \frac{c_{d,k,\alpha,f}}{\mathcal{N}^{\alpha-k/(d+1)}}. 
\end{split}
\end{align}
In particular, for $k=0$ we find that
\begin{equation}
     \norm{f-\widehat{f}^\mathcal{N}}_{L^{\infty}([0,1]^d)} \leq  (1+\delta)Q\cdot  \exp\left(-\alpha\mathcal{N}^{\frac{1}{d+1}} \log(\mathcal{N})\right) \leq \frac{(1+\delta)Q}{\mathcal{N}^{\alpha}}.
\end{equation}
Using Lemma \ref{lem:size-pqn}, we find that the network widths are respectively $O((e\alpha d)^{d+1}\mathcal{N})$ and $O(d5^d\mathcal{N}^{\frac{d}{d+1}})$ for large $\mathcal{N}$ and $d$ (the exact sizes can be easily calculated). 
\end{proof}

We thus find that tanh neural networks with two hidden layers result in an exponential convergence rate.
Moreover, the above corollary shows that a convergence rate that is independent of the dimension can be obtained, thereby lessening the curse of dimensionality. The proof however shows that even though the rate is free of the curse of dimensionality, the constant implied in the Landau notation still depends (super)exponentially on the dimension. Similar papers observe the same phenomena \cite{opschoor2019exponential}, or do not discuss this. 

\begin{remark}
One can also restate the previous corollary by saying that an approximation rate of $O(\mathcal{N}^k\exp\left(-\mathcal{N}\right))$ can be obtained using a tanh neural network with two hidden layers of widths $O\left(\mathcal{N}\binom{\mathcal{N}+d}{\mathcal{N}}\right)$ and $O(1)$ for $\mathcal{N}\to \infty$. Since $O\left(\mathcal{N}\binom{\mathcal{N}+d}{\mathcal{N}}\right)$ grows asymptotically slower than $O(\mathcal{N}^{d+1})$, another (very modest) lessening of the curse of dimensionality is revealed.
\end{remark}

Next, we show that, under an additional assumption, shallow tanh neural networks can also approximate analytic functions at an exponential rate. Moreover, in contrast to Corollary \ref{cor:analytic2}, there are no hidden constants that grow as $O(d^d)$. For simplicity, we restrict ourselves to approximation in supremum norm (i.e. $k=0$).

\begin{corollary}\label{cor:analytic3}
Let $d\in \mathbb{N}$, $\Omega\subset \mathbb{R}^d$ open with $[0,1]^d \subset \Omega$ and let f be analytic on $\Omega$. If $f$ satisfies for some $C>0$ that $\abs{f}_{W^{s,\infty}([0,1]^d)}\leq C^s$ for all $s\in \mathbb{N}$,
then for every $\mathcal{N}\in\mathbb{N}$ there exists a shallow tanh neural network $\widehat{f}^\mathcal{N}$ of width $3\left\lceil\frac{\mathcal{N}+5Cd}{2}\right\rceil \binom{\mathcal{N}+(5C+1)d}{\mathcal{N}+5Cd}$ (or $3\left\lceil\frac{\mathcal{N}}{2}\right\rceil$ for $d=1$) such that 
\begin{equation}
     \linfunit{f-\widehat{f}^\mathcal{N}} \leq  \exp(-\mathcal{N}).
\end{equation}
\end{corollary}
\begin{proof}
Assume that $f$ satisfies for some $C>0$ that $\abs{f}_{W^{s,\infty}([0,1]^d)}\leq C^s$ for all $s\in \mathbb{N}$. We calculate that for $\rho>1$,
\begin{equation}
    \frac{C^s}{s!}\left(\frac{3d\rho}{2}\right)^s =  \frac{1}{s!}\left(\frac{3Cd\rho}{2}\right)^s \leq \frac{1}{\sqrt{2\pi}} \left(\frac{3Cde\rho}{2s}\right)^s \leq \frac{1}{\sqrt{2\pi}} e^{3Cd\rho/2}, 
\end{equation}
where we used Stirling's approximation and maximized over all $s$. This proves that $f$ is $(Q,R)$-analytic with $Q=\frac{1}{\sqrt{2\pi}} e^{3Cd\rho/2}$ and $R=\frac{3d\rho}{2}$. Using Corollary \ref{cor:analytic1} with $s=\mathcal{N}$ and $N=1$ gives us that
\begin{equation}
     \linfunit{f-\widehat{f}^\mathcal{N}} \leq 2  \frac{1}{\sqrt{2\pi}} e^{3Cd\rho/2} \rho^{-\mathcal{N}}.
\end{equation}
If we set $\rho=e$, then $e^{3d\rho/2}\leq e^{5d}$. Therefore it holds that
\begin{equation}
     \linfunit{f-\widehat{f}^\mathcal{N}} \leq \exp(-\mathcal{N}+5Cd). 
\end{equation}
Note that since now $N=1$, the network architecture is even simpler: there is no need to construct a partition of unity, nor does there need to be a second hidden layer in order to approximately multiply the results of subnetworks. Therefore, a shallow tanh neural network with $3\left\lceil\frac{\mathcal{N}}{2}\right\rceil\abs{P_{\mathcal{N}-1,d+1}} \leq 3\left\lceil\frac{\mathcal{N}}{2}\right\rceil \binom{\mathcal{N}+d}{\mathcal{N}}$ neurons in its hidden layer suffices. The statement from the theorem is obtained by making the substitution $\mathcal{N}\leftarrow \mathcal{N}+5Cd$.
\end{proof}

Finally, we discuss how dimension-independent convergence rates can be obtained for a class of countably-parametric, holomorphic maps $f:U:=[-1,1]^\mathbb{N}\to\mathbb{R}$, which arise in e.g. elliptic PDEs with uncertain coefficients. This was first discussed in \cite{schwabzech2019} for deep ReLU neural networks and we will show that their results can be adapted to hold for shallow tanh neural networks. More precisely, their results hold for functions $u$ that admit a representation as a sparse Taylor generalized polynomial chaos expansion
\begin{equation}
    f(y) = \sum_{\nu\in\mathcal{F}} \frac{D^\nu f(0)}{\nu !} y^\nu, 
\end{equation}
which is unconditionally convergent for $y\in U$ and where $\mathcal{F}$ is defined by
\begin{equation}
    \mathcal{F}
=
\{
{\nu} \in \mathbb{N}_0^\mathbb{N}
\;
|
\;
\nu_j \ne 0 \text{ for only finitely many $j$}
\}.
\end{equation}
For a multi-index $\nu\in \N_0^\N$, we denote by $\mathrm{supp}(\nu) = \{j\in \N \;|\;\nu_j\ne 0\}$ the support of $\nu$, and we denote by $|\nu|=\sum_{j\in \N} |\nu_j|$ the $\ell^1$-norm of $\nu$.

It is shown in \cite[Section 2]{schwabzech2019} that $f$ admits such a representation if $f$ is $(b,\epsilon)$-holomorphic for $b\in \ell^p(\mathbb{N})$, $p\in (0,1]$ and $\epsilon>0$. The notion of $(b,\epsilon)$-holomorphy is defined as follows. 

\begin{definition}[Def. 2.1 in \cite{schwabzech2019}]\label{def:b-eps-holomorphy}
Let $V$ be a Banach space. Let $b\in \ell^p(\mathbb{N})$, $p\in (0,1]$ be a monotonically decreasing sequence. A poly-radius $\rho \in [1,\infty)^\mathbb{N}$ is called $(b,\epsilon)$-admissible for some $\epsilon>0$ if
\begin{equation}
    \sum_{j\in\mathbb{N}} b_j (\rho_j-1) \leq \epsilon.
\end{equation}
A continuous function $f : U \to V$ is called $(b,\epsilon)$-holomorphic if there exists a constant $C_f<\infty$ such that the following holds:
For every $(b,\epsilon)$-admissible $\rho$, there exists an extension $\Tilde{f} : B_\rho \to V_{\mathbb{C}}$ of $f$, i.e. we have $\Tilde{f}(y) = f(y)$ for all $y\in U\subset B_\rho$, $ \Tilde{f}$ is holomorphic in each component and such that $\sup_{z\in B_\rho} \norm{\Tilde{f}(z)}_{V_{\mathbb{C}}}\leq C_f$. Here, $B_\rho \subset \mathbb{C}^{\mathbb{N}}$ denotes the ball of polyradius $\rho$:
\[
B_\rho
=
\{
z \in \mathbb{C}^\N
\;
|
\;
|z_j| < \rho_j, \; \forall\, j\in \N
\},
\]
and $V_{\mathbb{C}} \simeq V + iV$ is the complexification of $V$.
\end{definition}

In \cite{schwabzech2019}, it is shown that for a $(b,\epsilon)$-holomorphic function $f$, an approximation rate of $O(\mathcal{N}^{1-1/p})$ can be obtained using a ReLU neural network of depth $O(\log(\mathcal{N})\log\log(\mathcal{N}))$ or using a neural network with a smoother activation function of depth $O(\log\log(\mathcal{N}))$. We show that their construction can also be used to obtain a dimension-independent approximation rate for shallow tanh neural networks. 

\begin{corollary}\label{cor:b-eps-hol}
Let $f:U=[-1,1]^\mathbb{N}\to\mathbb{R}$ be $(b,\epsilon)$-holomorphic for $b\in \ell^p(\mathbb{N})$, $p\in (0,1)$ and $\epsilon>0$. Then there exists a constant $C>0$ such that for every $\mathcal{N}\in\mathbb{N}$ there exists a shallow tanh neural network $\widehat{f}^\mathcal{N}$ of width at most $O\left(\mathcal{N} (C\log(\mathcal{N}))^{C\log(\mathcal{N})}\right)$ such that
\begin{equation}
     \norm{f-\widehat{f}^\mathcal{N}}_{L^\infty(U)} = O\left(\mathcal{N}^{1-1/p}\right) \qquad \text{ for } \mathcal{N}\to \infty.
\end{equation}
\end{corollary}

\begin{proof}
There exist a constant $C>0$ and a sequence of index sets $(\Lambda_\mathcal{N})_{\mathcal{N}\in\mathbb{N}} \subset \mathcal{F}$ for which it holds that (cf. \cite[proof of Thm. 3.9]{schwabzech2019})
\begin{equation}
    \sup_{y \in U}\abs{f(y)-\sum_{\nu\in\Lambda_\mathcal{N}} \frac{D^\nu f(0)}{\nu !} y^\nu} = O\left(\mathcal{N}^{1-1/p}\right).
\end{equation}
and such that $\abs{\Lambda_\mathcal{N}}=\mathcal{N}$, supp$(\nu) \subseteq \{1,\ldots, \mathcal{N}\}$ for all $\nu \in \Lambda_\mathcal{N}$ and for all $\mathcal{N}$ \cite[proof of Thm. 3.9]{schwabzech2019}, and $\sup_{\mathcal{N}\in\mathbb{N}} \abs{\nu} \leq C(1+\log(\mathcal{N}))$, where $\abs{\nu} := \sum_{j\in \N} |\nu_j|$ \cite[Thm. 2.7]{schwabzech2019}. The latter implies in particular that $\sup_{\mathcal{N}\in\mathbb{N}} \abs{\text{supp}(\nu)} \leq C(1+\log(\mathcal{N}))$.

Based on these results from \cite{schwabzech2019}, it therefore suffices to show that we can accurately approximate all monomials $y\mapsto y^\nu$ for $\nu \in \Lambda_\mathcal{N}$ with shallow tanh neural networks. For a fixed $\nu \in \Lambda_\mathcal{N}$, the monomial $y\mapsto y^\nu$ can be approximated (to arbitrary accuracy) using Corollary \ref{cor:mon-mult1} with $d=n\leftarrow C(1+\log(\mathcal{N}))$, resulting in a shallow tanh neural network of width $O\left((Ce(1+\log(\mathcal{N})))^{C(1+\log(\mathcal{N}))+1}\right)$. The network $\hat{f}^\mathcal{N}$ from the statement can then be constructed by parallelizing all the networks that approximate the individual monomials, yielding an approximation 
\[
\left\Vert f - \hat{f}^\mathcal{N} \right\Vert_{L^\infty(U)}
= 
O(\mathcal{N}^{1-1/p}).
\]
To be precise, we take the input of this network to be $(y_1, \ldots y_\mathcal{N})$ instead of $y$, which is possible since supp$(\nu) \subseteq \{1,\ldots, \mathcal{N}\}$ for all $\nu \in \Lambda_\mathcal{N}$. As $\abs{\Lambda_\mathcal{N}}=\mathcal{N}$, the resulting width of $\hat{f}^\mathcal{N}$ is $O\left(\mathcal{N} (Ce(1+\log(\mathcal{N})))^{C(1+\log(\mathcal{N}))+1}\right)$, which is asymptotically equivalent to the width from the statement for $\mathcal{N}\to \infty$.
\end{proof}

This result implies in particular, that linear functionals of parametric solutions of PDEs can be approximated by shallow tanh neural networks \cite{schwabzech2019}. Following \cite[Section 4]{schwabzech2019}, the result can also be extended to directly approximate the parametric solution manifold, e.g. to approximate $(b,\epsilon)$-holomorphic operators of the form $f:[-1,1]^\mathbb{N}\to H^1_0([0,1])$.

\subsection{Examples}\label{sec:ex}

In this section, we illustrate the bounds derived in Theorem \ref{thm:main} with some prototypical examples. In particular, we will investigate the width, weights and sparsity of the networks from the proof of Theorem \ref{thm:main}. 

First, we demonstrate how large the networks of Theorem \ref{thm:main} are for a simple function approximation example with $d=1$ and {\color{black} $L^{\infty}$-norm}. We consider the functions
\begin{equation}
    f_a : [0,1]\to [-1,1]: x \mapsto \sin(ax), \quad a>0.
\end{equation}
For a given error tolerance $\varepsilon>0$, we look for a three-layer tanh neural network $\widehat{f}^{N,s}$, as given by Theorem \ref{thm:main}, such that provably
\begin{equation}
    \norm{f_a-\widehat{f}^{N,s}}_\infty \leq \varepsilon.
\end{equation}
From all the networks that satisfy this condition, we take the one with the minimal width. More rigorously, we select
\begin{equation}
    \text{argmin}_{s,N\in\mathbb{N}\::\: (3a/2N)^s/s! < \varepsilon} \max\left\{3\left\lceil\frac{s}{2}\right\rceil+N-1,6N\right\},
\end{equation}
where we used that $\abs{f_a}_{W^{s,\infty}}\leq a^s$ for $s\geq 1$. Alternatively, one can also set $N=1$ in Theorem \ref{thm:main}, which makes the bound more efficient as no more partition of unity is needed, thereby reducing the need for a second hidden layer. This is similar to the proof of Corollary \ref{cor:analytic1}. In this case, we select
\begin{equation}
    \text{argmin}_{s\in\mathbb{N}\::\: (a/2)^s/s! < \varepsilon} \left\{3\left\lceil\frac{s}{2}\right\rceil\right\}.
\end{equation}
We present the result in Figure \ref{fig:exp1}. For the chosen examples, a shallow (i.e. two-layer) tanh neural network achieves a similar level of error as a three-layer network of the same width. This can be explained by the fact that $\abs{f_a}_{W^{s,\infty}}$ grows as $O(a^s)$ and not as $O(a^s s!)$, such that setting $N>1$ is not required for the bound of Theorem \ref{thm:main} to be non-vacuous. Moreover, the networks suggested by Theorem \ref{thm:main} are not unreasonably large for this simple example. Yet, they still remain overestimates: we found that e.g. $f_{2\pi}$ can already be approximated to an error of $1\%$ by a shallow tanh network of width four. Finally, the exponential convergence is evident as a small increase in the network width already leads to a very large improvement in the accuracy. 

\begin{figure}
    \centering
    \includegraphics[width=0.8\textwidth]{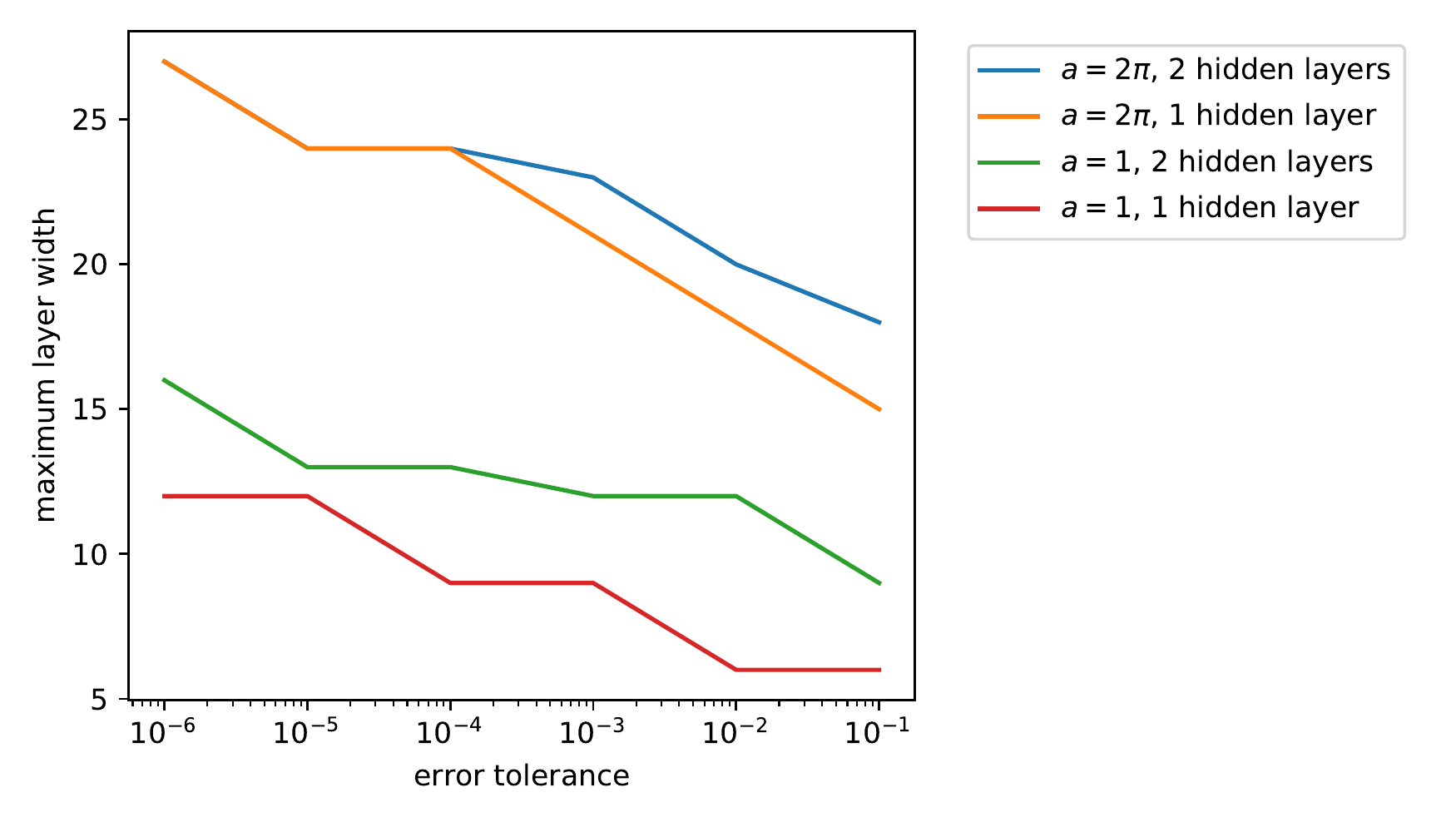}
    \caption{Needed layer width according to Theorem \ref{thm:main} to approximate the function $f_a$ to a given error tolerance. }
    \label{fig:exp1}
\end{figure}

Next, we investigate whether the blow-up of the network weights from the theoretical results is observed in practice. We approximate univariate monomials of odd power in supremum norm on the interval $[0,1]$ using shallow neural networks whose sizes are determined by Lemma \ref{lem:uni-mon-sobolev}. We generate a training set using $2000$ randomly generated points, based on the uniform distribution on $[0,1]$ and minimize the training loss for 2000 epochs using the Adam optimizer \cite{kingma2014adam}. The results can be found in Table \ref{tab:exp2} and show that the weights do not blow up in practice. Rather, the weights remain small for this example. This is possibly a consequence of the phenomenon of implicit regularization in deep learning, e.g. \cite{neyshabur2014search}. 

\begin{table}[]
    \centering
    \begin{tabular}{|c|c|c|c|c|c|}\hline
    power &  1 & 3 & 5 & 7 & 9\\\hline
    MSE    & $4.91 \cdot 10^{-7}$ & $1.54\cdot 10^{-6}$ & $2.87\cdot 10^{-5}$ & $1.13\cdot 10^{-4}$ & $6.13\cdot 10^{-5}$\\\hline
    largest weight & $3.13$ & $2.14$ & $2.27$ & $4.55$ & $4.41$\\\hline
    \end{tabular}
    \caption{MSE and largest weight (in absolute value) of shallow tanh neural networks that approximate univariate monomials with odd powers on $[0,1]$.}
    \label{tab:exp2}
\end{table}

Finally, we show that the neural networks constructed in the proof of Theorem \ref{thm:main} are not very sparse i.e., the fraction of non-zero weights of the network, compared to the total number of weights, is not small. Figure \ref{fig:exp3} shows that the fraction of non-zero weights of the network increases with increasing $s$ and decreasing $N$ (for $d=1$). For analytic functions, it is (asymptotically) more efficient to increase $s$ than $N$, as the convergence rate is $O(N^{-s})$. This lets us conclude that the constructed networks corresponding to sensible choices of $s$ and $N$ are, in general, quite dense. This is in agreement with what one observes in practice. This is in contrast to the theoretical results for deep ReLU (and other) neural networks, where the sparsity of the constructed networks generally increases with increasing accuracy \cite{yarotsky2017error,opschoor2019exponential,guhring2021approximation}. 

\begin{figure}
    \centering
    \includegraphics[width=0.8\textwidth]{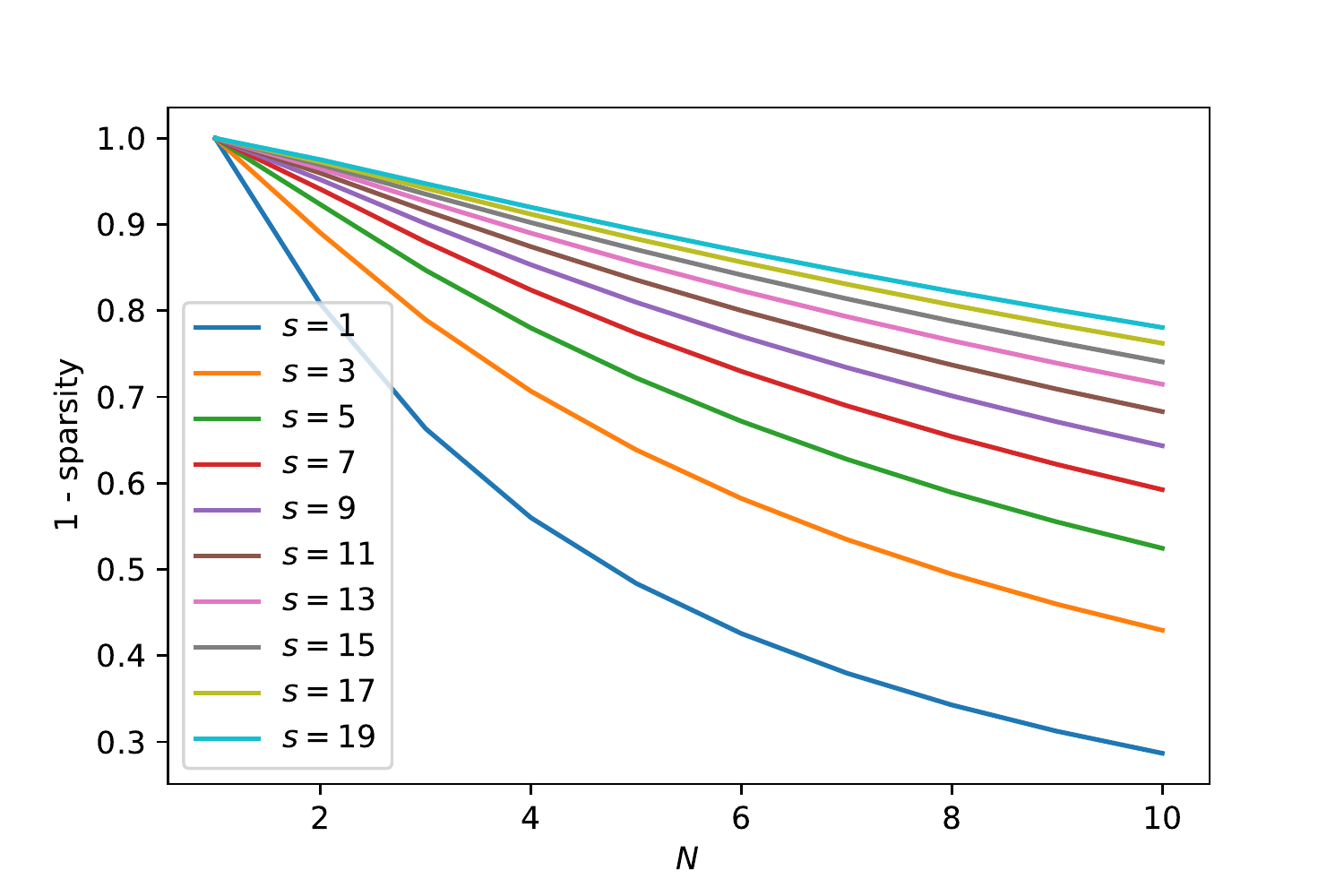}
    \caption{Fraction of non-zero weights (i.e. $1-$ sparsity) of the networks of Theorem \ref{thm:main} for different values of $s$ and $N$ in the case where $d=1$. }
    \label{fig:exp3}
\end{figure}

\section{Summary and discussion}
\label{sec:6}
The main aim of this paper was to provide explicit bounds on the error (in high-order Sobolev) norms with which a neural network with a tanh activation function approximates Sobolev-regular functions, $C^k$ functions and analytic functions. To this end, we prove such explicit bounds on the approximation error for Sobolev functions in Theorem \ref{thm:main} and for analytic functions in Corollary \ref{cor:analytic1}. In both cases, we prove these bounds for a tanh neural network with just $2$ hidden layers. Our proofs are constructive and the construction relies on three key ideas: (1) the approximation of monomials by finite differences of a smooth activation function (Lemma \ref{lem:monomial-sobolev}), (2) the approximation of the multiplication operator (Lemma \ref{cor:mult-shallow}) and (3) the approximation of a partition of unity (Section \ref{sec:pou}). In particular, we prove that a neural network with only hidden layers and a tanh activation function yields the same (or better) approximation rates for Sobolev-regular and for analytic functions.

We elaborate this point further by comparing and contrasting our approach and results with the large body of literature on approximation of functions with artificial neural networks. 

First, we compare our approach, as stated above, with other related works. The simple, yet very effective trick of approximating monomials by finite differences of smooth activation function has been around for decades \cite{pinkus1999approximation}, but is still a building block in the constructive proofs of many recent papers on neural network expressivity, e.g. \cite{rolnick2017power,ohn2019smooth,siegel2020approximation,guhring2021approximation}. To the best of the authors' knowledge, all available results build upon the observation that there is a $x\in \mathbb{R}$ such that $\sigma^{(n)}(x) \neq 0$ for all $n\in\mathbb{N}$. As such, this construction does not allow for explicit estimates on the approximation error and the network weights (see Remark \ref{rem:why-odd}). \emph{Our key novelty in this paper is to circumvent this issue by first approximating univariate monomials of odd powers and then expanding to even powers and multivariate monomials.} This allows us to obtain \emph{uniform} explicit bounds for the error in approximating multivariate monomials of a varying degree and paves the way for explicit bounds on the approximation error. 

The approximation of the multiplication operator in $d$ dimensions by a shallow neural network (Corollary \ref{cor:mult-shallow}) was discussed in \cite[Appendix A]{lin2017does} for activation functions $\sigma$ that satisfy that $\sigma(x) = \sum_k \sigma_k x^k$ where $\sigma_k \neq 0$ for $0 \leq k \leq d$. In particular, they prove that $2^d$ neurons are both sufficient and necessary. However, this construction does not allow for explicit estimates (again cf. Remark \ref{rem:why-odd}). Here, we propose a \emph{novel construction} of the multiplication operator with a shallow tanh neural network. 

As for the partition of unity, which serves as an essential ingredient in our proofs, an exact partition of unity for ReLU neural networks can be readily constructed \cite{yarotsky2017error}. Approximations of partitions of unity with neural networks with sigmoidal activation function can be found in \cite{costarelli2013approximation,costarelli2013multivariate,ohn2019smooth} and a general framework for approximations of partitions of unity was proposed in \cite{guhring2021approximation}. In this paper, we have constructed approximations of partitions of unity with shallow tanh neural networks, that were motivated by localized polynomials which arise in the Bramble-Hilbert Lemma and the Taylor's theorem. Compared to the other works mentioned above, our results on partitions of unity stand out for the explicit bounds on the approximation error and the weights. 

Given the afore-mentioned novel ideas, we were able to obtain explicit bounds on the approximation error. A suitable avenue to compare our results with results obtained in related works lies in the approximation error bounds for analytic functions. We recall that we prove approximation rates to analytic functions in the $W^{k,\infty}$-norm. Although approximation rates in this norm were proved for Sobolev functions in the very recent paper \cite{guhring2021approximation}, it is unclear if their results can be extended in an efficient way for analytic functions. A key reason for this lies in the fact that the widths of their constructed networks are not explicitly stated and the depth increases with maximal degree of monomials, inhibiting uniform control that is necessary for approximating analytic functions. 

Exponential convergence (in terms of network size) of neural networks for analytic functions in the $L^{\infty}$-norm was first proven in \cite{mhaskar1996neural} for neural networks with smooth activation functions and in \cite{wang2018exponential} for ReLU neural networks. In \cite{opschoor2019exponential,herrmann2021constructive}, the authors prove exponential convergence in $W^{1,\infty}${-norm} for ReLU neural networks. We compare our results for approximation of analytic functions with these papers in Table \ref{tab:comparison}. For \cite{mhaskar1996neural}, the parameter $\rho$ is related to the polyradius of the ellipse to which the function needs to be holomorphically extendable. In \cite{wang2018exponential}, the additional assumption is made that the analytic function admits a Taylor expansion on $[-1,1]^d$ that converges absolutely and uniformly, which does not hold for general analytic functions. For \cite{opschoor2019exponential}, the parameter $\beta$ is at least inversely proportional to the dimension $d$ and also depends on the radius of the Bernstein ellipse to which the analytic function can be holomorphically extended. From Table \ref{tab:comparison}, one can clearly observe that Corollary \ref{cor:analytic2} yields an asymptotically faster convergence in terms of network width than the other related works. In addition, our results hold in stronger norms and we provide explicit bounds on the approximation error and weights, in contrast to other papers.  For instance in \cite{mhaskar1996neural,opschoor2019exponential}, the convergence rate even depends on the (unknown) polyradius of the ellipse to which the function can be holomorphically extended. Lastly, note that Corollary \ref{cor:analytic3} assumes that $\abs{f}_{W^{s,\infty}([0,1]^d)}\leq C^s$ for some $C>0$ and all $s\in \mathbb{N}$, which implies that $f$ must be entire. In \cite[Theorem 5.4]{HSZ20_2810},  exponential expressivity of ReLU neural networks for entire functions is proven. Compared to Corollary \ref{cor:analytic3}, this result is more efficient in terms of non-zero weights, but requires $O\left(\log(1/\epsilon)\log\log(1/\epsilon)\right)$ layers to obtain an approximation with accuracy $\epsilon>0$.   

All of our approximation results, including the approximation bounds on analytic functions hold for a tanh neural network with only two hidden layers. This result, see also \cite{mhaskar1996neural}, runs contrary to the prevailing view that depth of neural networks is essential for function approximation and establishes that shallow but wide neural networks can be very expressive when it comes to function approximation and might provide some justification for the use of very shallow and wide neural networks in scientific computing \cite{LSK1,DeepOnets,MM3}. 

Finally, it is essential to mention that although we highlight our contribution in terms of the tanh activation function as it is the most commonly used of the smooth activation functions. Our results apply verbatim to the logistic or sigmoid activation function as it is a shifted and scaled tanh. However, our constructions also apply to a much larger class of smooth activation functions as elaborated in Section \ref{sec:3}.  

\begin{table}[]
    \centering
    \resizebox{\textwidth}{!}{%
    \begin{tabular}{|c|c|c|c|c|c|}
    \hline
    source & norm & activation &  depth & width& error bound \\\hline
    \cite[Thm 2.3]{mhaskar1996neural} & $L^\infty([-1,1]^d)$ & $C^\infty$ & 2 & $\mathcal{N}$ & $O\left(\rho^{-\mathcal{N}^{1/d}}\right)$ \\\hline
     \cite[Thm. 6]{wang2018exponential} &$L^\infty([-1+\delta, 1 - \delta]^d)$ & ReLU & $O(\mathcal{N})$ & $d+4$ & $O\left(\exp(-d\delta \mathcal{N}^{1/2d})\right)$\\\hline
     \cite[Thm. 3.6]{opschoor2019exponential} & $W^{1,\infty}([-1,1]^d)$ & ReLU & $O(\mathcal{N}^{\frac{1}{d+1}}\log(\mathcal{N}))$ & $O(\mathcal{N})$ & $O\left(\exp(-\beta \mathcal{N}^{\frac{1}{d+1}})\right)$ \\\hline
     this work & $W^{k,\infty}([0,1]^d)$ & tanh & 3 & $O(\mathcal{N})$ & $O\left(\mathcal{N}^{\frac{k}{d+1}}\exp(- \mathcal{N}^{\frac{1}{d+1}}\ln(\mathcal{N}))\right)$ \\\hline
    \end{tabular}
    }
    \caption{Comparison of upper bounds on the approximation error for analytic functions by neural networks. }
    \label{tab:comparison}
\end{table}
We conclude by pointing out some limitations of the presented results. The most important limitation is the fact that the amplitude of the weights in our constructive network can grow very fast (Theorem \ref{thm:main}). In practice, implicit and explicit regularization mechanisms during training will ensure that such growth of weights will not happen. In fact, we present examples to empirically show that the gradient-descent based training procedure manages to find rather small weights and biases that still provide a very high accuracy. We therefore believe that our bounds are useful in practice, more as upper bounds for setting the network size.  

Another limitation, which we share with other published results on approximation with neural networks, is that our results suffer from the curse of dimensionality. We could however prove that it is possible to obtain a dimension-independent convergence rate to analytic functions in Corollary \ref{cor:analytic2}. Another possible mitigation of the curse of dimensionality for the approximation rate is when the underlying map is $(b,\epsilon)$-holomorphic, see Corollary \ref{cor:b-eps-hol} for the precise result. However in these cases, the constants (and hence the network size) can still grow exponentially in the input dimension.  Fortunately, one can argue that a large number of high-dimensional functions are in fact compositions of low-dimensional functions, which might explain the success of deep learning in high dimensions \cite{poggio2017and}. For instance, the Kolmogorov-Arnold superposition theorem \cite{kolmogorov1957representation} even states that all $d$-variate functions are in fact compositions of univariate functions and the sum of $d$ numbers, which also can be used to lessen the curse of dimensionality \cite{montanelli2020error}. 

Finally, the weights in our constructed networks are continuous with respect to the function of interest $f$. It has been proven that the best neural approximation cannot be achieved using continuous weight selection \cite{kainen1999approximation}. An example of how discontinuous weight selection can improve the approximation rate can be found in \cite{yarotsky2018optimal}. 

\section*{Acknowledgements}
SL and SM received funding from the European Research Council (ERC) under the European Union’s Horizon 2020 research and innovation programme (grant agreement No. 770880)

\appendix

\section{Auxiliary results}

\begin{lemma}\label{lem:tanh-derivatives}
It holds for every $n\in\mathbb{N}$ that $\abs{\tanh^{(2n-1)}{(0)}} \geq 1$. 
\end{lemma}

\begin{proof}
For $\abs{x}<\frac{\pi}{2}$, the power series expansion of tanh at $x$ is given by
\begin{equation}
    \tanh{(x)} = \sum_{n=1}^\infty \frac{2^{2n}(2^{2n}-1)B_{2n}}{(2n)!}x^{2n-1}, 
\end{equation}
where $B_n$ is the $n$-th Bernoulli number. One can then calculate that for every $n\in\mathbb{N}$, 
\begin{equation}\label{eq:lb-tanh0}
    \abs{\tanh^{(2n-1)}{(0)}} = \abs{\frac{2^{2n}(2^{2n}-1)B_{2n}}{2n}} \geq 1. 
\end{equation} 
This concludes the proof of the statement. 
\end{proof}

\begin{lemma}\label{lem:alpha}
Let $s\in 2\mathbb{N}-1$. It holds that
\begin{equation}
    \inf_{\alpha>0} \frac{2^{s/2}(1+\alpha)^{(s^2+s)/2}}{\alpha^{s/2}} \leq \sqrt{e} (2es)^{s/2},  
\end{equation}
where the infimum is reached at $\alpha=1/s$.
\end{lemma}
\begin{proof}
Some elementary calculations show that
\begin{equation}
     \inf_{\alpha>0} \frac{2^{s/2}(1+\alpha)^{(s^2+s)/2}}{\alpha^{s/2}} = 2^{s/2}\left(1+\frac{1}{s}\right)^{(s^2+s)/2} s^{s/2}. 
\end{equation}
Moreover, it holds that
\begin{equation}
    \left(1+\frac{1}{s}\right)^{(s^2+s)/2} = \left(\left(1+\frac{1}{s}\right)^s\right)^{(s+1)/2} \leq e^{(s+1)/2}.
\end{equation}
The statement follows immediately from these inequalities. 
\end{proof}

\begin{lemma}\label{lem:dysoninv}
Let $n,q\in\mathbb{N}$ and $P_{n,q} = \{\alpha\in  \mathbb{N}_0^q : \abs{\alpha} = n\}$. If $D = (D_{\alpha,\beta})_{\alpha,\beta\in P_{n,q}}$ is defined as in \eqref{eq:dyson}, then $D$ is invertible and 
\begin{equation}
    \norm{D^{-1}}_\infty \leq (n!)^3 \abs{P_{n,q}}^2 2^n. 
\end{equation}
\end{lemma}
\begin{proof}
Following \cite{moak1990combinatorial}, let $P_{n,q} = \{\alpha\in  \mathbb{N}_0^q : \abs{\alpha} = n\}$ and $I_{n,q} = \{\alpha'\in  \mathbb{N}_0^{q-1} : \abs{\alpha'} \leq n\}$. Let $s(k,m)$ be Stirling numbers of the first kind, for $k \leq m$, defined by
\begin{equation}\label{eq:stirling-numberrs}
    x(x-1)\cdots(x-m+1) = \sum_{k=0}^m s(k,m) x^k. 
\end{equation}
For $\alpha', \beta' \in I_{n,q}$, define $S_{\alpha', \beta'} = \prod_{i=1}^{q-1} s(\alpha'_i, \beta'_i)$, where $S(\alpha', \beta') = 0$ unless $\alpha'_i\leq \beta'_i$ for all $i$.  Denote by $S$ the corresponding matrix, where the order of rows and columns reflects the lexicographic order on $I_{n,q}$. Next, define $B$ by
\begin{equation}
    B_{\alpha', \beta'} = \binom{\alpha'}{\beta'}(-1)^{\abs{\beta'}} := \prod_{i=1}^{q-1} \binom{\alpha'_i}{\beta'_i}(-1)^{\beta'_i} \quad \text{for }\alpha', \beta' \in I_{n,q}.
\end{equation}
Finally, let $L$ and $\Lambda$ be diagonal matrices defined by $L_{\alpha',\alpha'} = (-1)^{\abs{\alpha'}}/\alpha'!$ and $\Lambda_{\alpha',\alpha'} = n(n-1)\cdots (n-\abs{\alpha'}+1)/\alpha'!$ for $\alpha'\in I_{n,q}$. It then holds that \cite[Corollary 2]{moak1990combinatorial}, 
\begin{equation}\label{eq:dysoninv}
    D^{-1} = BL^{-1}\Lambda^{-1}SLB. 
\end{equation}
To prove an upper bound on the supremum norm of $D^{-1}$, we first note that $\abs{s(k,m)} \leq \sum_{k=0}^m \abs{s(k,m)} = m!$ (which can be seen by setting $x=-1$ in \eqref{eq:stirling-numberrs}) and thus $S(\alpha', \beta') \leq \beta'!$. In addition, it holds for any $\alpha'\in I_{n,q}$ that
\begin{equation}
    1 \leq \frac{n!}{\alpha'!(n-\abs{\alpha'})!} \leq \frac{n!}{\alpha'!}, 
\end{equation}
which gives us that $\max_{\alpha'\in I_{n,q}} (\alpha'!) \leq n!$. 
This gives us consequently
\begin{align}
    &\abs{L^{-1}\Lambda^{-1}S}_{\alpha', \beta'} \leq (\alpha'!)^2 \beta'! \leq (n!)^2 \beta'!, \\
    &\abs{L^{-1}\Lambda^{-1}SL}_{\alpha', \beta'} \leq (n!)^2, \\
    &\abs{L^{-1}\Lambda^{-1}SLB}_{\alpha', \beta'} \leq (n!)^2 \sum_{\gamma'\in I_{n,q}} \binom{\gamma'}{\beta'} \leq (n!)^3 \abs{I_{n,q}},\\
    &\abs{BL^{-1}\Lambda^{-1}SLB}_{\alpha', \beta'} \leq (n!)^3 \abs{I_{n,q}} \sum_{\gamma'\in I_{n,q}}\binom{\alpha'}{\gamma'} = (n!)^3 \abs{I_{n,q}} 2^{\abs{\gamma'}} \leq (n!)^3 \abs{I_{n,q}} 2^n.
\end{align}
This and \eqref{eq:dysoninv} let us conclude that $\norm{D^{-1}}_\infty \leq (n!)^3 \abs{I_{n,q}}^2 2^n$. The lemma then follows from the existence of a one-to-one correspondence of elements in $I_{n,q}$ and $P_{n,q}$. 
\end{proof}

\begin{lemma}\label{lem:bound-der-tanh}
Let $m\in \mathbb{N}$. Then it holds that 
\begin{equation}
    \abs{\sigma^{(m)}(x)} \leq (2m)^{m+1}\min\{\exp(-2x),\exp(2x)\} \quad \text{for all } x\in \mathbb{R}.
\end{equation}
\end{lemma}
\begin{proof}
In \cite{boyadzhiev2009derivative}, the following formula for the derivative of the hyperbolic tangent is proven, 
\begin{equation}\label{eq:der-tanh}
    \sigma^{(m)}(x) = (-2)^m (\sigma(x)+1) \sum_{k=0}^m \frac{k!}{2^k}{m \brace k}(\sigma(x)-1)^k, 
\end{equation}
where ${m \brace k}$ denote Stirling numbers of the second kind, for which it holds that ${m \brace k}\leq \frac{k^m}{k!}$. This then gives us
\begin{equation}\label{eq:tanh-bound1}
    \abs{\sigma^{(m)}(x)} \leq 2^m \abs{1+\sigma(x)} \sum_{k=0}^m k^m \leq 2^m m^{m+1} \abs{1+\sigma(x)} \leq (2m)^{m+1} \exp(2x), 
\end{equation}
as $ \sum_{k=0}^m k^m \leq m\cdot m^m \leq  m^{m+1}$. 
Furthermore one can note that $\sigma^{(m)}(-x) = - \sigma^{(m)}(x)$, which gives us
\begin{equation}
    \abs{\sigma^{(m)}(x)} \leq 2^m m^{m+1}\abs{1-\sigma(x)} \leq (2m)^{m+1} \exp(-2x). 
\end{equation}
The statement follows easily. 
\end{proof}

\begin{lemma}\label{lem:alpha-growth}
The conditions stated in \eqref{eq:alpha} for $k>0$ are satisfied if
\begin{equation}
    \alpha = N \max\left\{R,\ln(\frac{(2k)^{k+1}(Nk)^k}{e^k\epsilon})\right\}.
\end{equation}
\end{lemma}

\begin{proof}
The first condition of \eqref{eq:alpha} is trivially satisfied when $\alpha$ is chosen as in the statement. From Lemma \ref{lem:bound-der-tanh}, it follows that 
\begin{equation}
    \alpha^k (2k)^{k+1}\exp(-2\alpha/N) \leq \epsilon
\end{equation}
is a sufficient condition that implies the other conditions of \eqref{eq:alpha}. Using $\max_{\alpha>0} \alpha^k \exp(-\alpha/N) \leq (Nk)^k \exp(-k)$ we find that
\begin{equation}
    \alpha^k \exp(-2\alpha/N) = \alpha^k \exp(-\alpha/N) \exp(-\alpha/N) \leq (Nk)^k \exp(-k) \exp(-\alpha / N).
\end{equation}
The statement follows directly. 
\end{proof}

\begin{lemma}\label{lem:leibniz}
Let $d\in\mathbb{N}$, $k\in\mathbb{N}_0$, $\Omega\subset \mathbb{R}^d$ and $f,g\in W^{k,\infty}(\Omega)$. Then it holds that
\begin{equation}
    \ck{fg}\leq 2^k \ck{f}\ck{g}.
\end{equation}
\end{lemma}
\begin{proof}
The statement follows directly from the general Leibniz rule.
\end{proof}

\begin{lemma}\label{lem:faa-di-bruno}
Let $d,m,n\in\mathbb{N}$, $\Omega_1\subset \mathbb{R}^d$, $\Omega_2\subset \mathbb{R}^m$, $f\in C^n(\Omega_1; \Omega_2)$ and $g\in C^n(\Omega_2; \mathbb{R})$. Then it holds that 
\begin{equation}
    \norm{g \circ f}_{W^{n,\infty}} \leq 16(e^2n^{4}md^2)^{n} \norm{g}_{W^{n,\infty}} \max_{1\leq i\leq m}\norm{(f)_i}_{W^{n,\infty}}^n.
\end{equation}
\end{lemma}
\begin{proof}
Let $\nu\in\mathbb{N}^d$ with $\abs{\nu}=n$. We use the multivariate Faà di Bruno formula \cite{constantine1996multivariate}, 
\begin{equation}
    D^\nu (g\circ f) = \sum_{1\leq \abs{\lambda}\leq n} D^\lambda g \sum_{p(\nu,\lambda)} (\nu !) \prod_{j=1}^n  \frac{(f_{l_j})^{k_j}}{k_j! (l_j!)^{\abs{k_j}}},
\end{equation}
where $(f_\mu)_i = D^{\mu} (f)_i$ for $1\leq i \leq m$ and the set $p(\nu,\lambda)$ is defined as
\begin{align}
    \begin{split}
p(\nu,\lambda) = \{&(\kappa,\ell):=(k_1,\ldots,k_n; l_1, \ldots, l_n):\text{ for some } 1\leq s\leq n, \\
&k_i=0 \text{ and } l_i=0 \text{ for } 1\leq i\leq n-s; \abs{k_i}>0 \text{ for } n-s+1\leq i\leq n; \\
&\text{ and } 0 \prec l_{n-s+1} \prec \cdots \prec l_n \text{ are such that }\\
&\sum_{i=1}^n k_i=\lambda, \sum_{i=1}^n \abs{k_i}l_i = \nu
\}, 
    \end{split}
\end{align}
where $a\prec b$ either means that $\abs{a}<\abs{b}$ or $a<b$ according to lexicographic ordering; furthermore the vectors $k_i$ are $m$-dimensional and the $l_i$ are $d$-dimensional. From the stated conditions, it follows directly that $\sum_{i=1}^n\abs{k_i}\leq n$ and $\sum_{i=1}^n\abs{l_i}\leq n$. Next, we bound the complexity of $p(\nu,\lambda)$. From $\sum_{i=1}^n\abs{k_i}\leq n$, it follows that the number of $\kappa$ is bounded above by $\abs{P_{n,(m+1)n}}$, which can in turn be bounded by $\sqrt{\pi}e^n(mn)^n$ by Lemma \ref{lem:size-pqn}. Similarly, it follows that the number of $\ell$ is bounded above by $\abs{P_{n,(d+1)n}}$, which can in turn be bounded by $\sqrt{\pi}e^n(dn)^n$ by Lemma \ref{lem:size-pqn}. Therefore, $\abs{p(\nu,\lambda)} \leq \pi (e^2n^2md)^n$. Finally, we can make the estimates that $\abs{\{\lambda: 1\leq \abs{\lambda}\leq n\}}\leq \abs{P_{n,d+1}}\leq \sqrt{\pi}e^nd^n$, $D^\lambda g \leq \norm{g}_{W^{n,\infty}}$, $\nu! \leq n!$ and $\prod_{j=1}^n (f_{l_j})^{k_j} \leq \max_{1\leq i\leq m} \norm{(f)_i}_{W^{n,\infty}}^n$. Together with Stirling's approximation, this yields
\begin{align}
\begin{split}
     \norm{D^\nu (g\circ f)}_\infty &\leq \sqrt{\pi} e^{n} d^n  \norm{g}_{W^{n,\infty}} \cdot \pi (e^2n^2md)^n \cdot n! \cdot  \max_{1\leq i\leq m}\norm{(f)_i}_{W^{n,\infty}}^n \\ &\leq 16(e^2n^{4}md^2)^{n} \norm{g}_{W^{n,\infty}} \max_{1\leq i\leq m}\norm{(f)_i}_{W^{n,\infty}}^n. 
\end{split}
\end{align}
\end{proof}

\begin{lemma}[Bramble-Hilbert]\label{lem:bramble-hilbert2}
Let $\Omega\subset \mathbb{R}^d$ be an open and bounded set of diameter $0<h<e^{-1/2}d^{-3/2}$ which is star-shaped with respect to every point in an open ball $B\subset \Omega$ with diameter $\rho h$. Then for every $f\in W^{s,\infty}(\Omega)$ there exists a polynomial $\hat{f}$ of degree at most $s-1$ such that for any $k \in \mathbb{N}_0$ with $k<s$ it holds that, 
\begin{equation}
    \norm{f-\hat{f}}_{W^{k,\infty}(\Omega)} \leq \frac{\sqrt{s}\pi^{1/4}(d\sqrt{de} h)^{s-k}}{(s-k-1)!}   \abs{f}_{W^{s,\infty}(\Omega)}.
\end{equation}
\end{lemma}

\begin{proof}
By setting $p=q = \infty$ in the penultimate equation in the proof of the main theorem in \cite{duran1983polynomial} (note that this reference uses a different definition of Sobolev norm), it follows that there exists a polynomial $\hat{f}$ of degree at most $s-1$ such that for $0\leq m < s$ it holds that,
\begin{equation}
    \abs{f-\hat{f}}_{W^{m,\infty}(\Omega)} \leq (s-m) \left(\sum_{\beta\in P_{s-m,d}} (\beta !)^{-2}\right)^{1/2} h^{s-m} \sqrt{\abs{P_{s-m,d}}}\abs{f}_{W^{s,\infty}(\Omega)}.
\end{equation}
Using Lemma \ref{lem:size-pqn}, we find that
$\sqrt{\abs{P_{s-m,d}}} \leq \pi^{1/4}(ed)^{(s-m)/2}$
and from the multinomial theorem it follows that
\begin{equation}
    \sum_{\beta\in P_{s-m,d}} (\beta !)^{-2} \leq \sum_{\beta'\in P_{2(s-m),2d}} (\beta' !)^{-1} = \frac{(2d)^{2(s-m)}}{(2(s-m))!}. 
\end{equation}
One can also calculate that $(2(s-m))! \geq 4^{s-m} (s-m) ((s-m-1)!)^2$. Combining the previous observations, we find
\begin{equation}
     \abs{f-\hat{f}}_{W^{m,\infty}(\Omega)} \leq \frac{\sqrt{s-m}(d\sqrt{de} h)^{s-m}}{(s-m-1)!}   \abs{f}_{W^{s,\infty}(\Omega)}.
\end{equation}
Majorizing over $0\leq m \leq k$ then gives the upper bound from the statement. 

\end{proof}

\begin{lemma}[Taylor's theorem]\label{lem:taylor}
Let $d,s \in \mathbb{N}$, $0<\delta<1/d$. Then for every $f\in C^{s}([-\delta,\delta]^d)$ there exists a polynomial $\hat{f}$ of degree at most $s-1$ such that for any $k \in \mathbb{N}_0$ with $k<s$ it holds that, 
\begin{equation}
    \norm{f-\hat{f}}_{W^{k,\infty}([-\delta,\delta]^d)} \leq \frac{(d\delta)^{s-k}}{(s-k)!}\abs{f}_{W^{s,\infty}([-\delta,\delta]^d)}
\end{equation}
\end{lemma}
\begin{proof}
We give a constructive proof. For $f\in C^{s}([-\delta,\delta]^d)$, we define the polynomial $\hat{f}$ as
\begin{equation}
    \hat{f}(x) = \sum_{\abs{\alpha}\leq s-1} \frac{D^\alpha f(0)}{\alpha!} x^\alpha. 
\end{equation}
Then take $\beta\in\mathbb{N}_0^d$ with $\abs{\beta}\leq k$. It then holds that
\begin{equation}
    D^\beta \hat{f}(x) = \sum_{\substack{\abs{\alpha}\leq s-1\\ \alpha\geq \beta}} \frac{D^\alpha f(0)}{\alpha!}   \frac{\alpha!}{(\alpha-\beta)!}x^{\alpha-\beta} = \sum_{\abs{\gamma}\leq s-1-\abs{\beta}} \frac{D^\gamma D^\beta f(0)}{\gamma!} x^\gamma.
\end{equation}
For $x\in[-\delta,\delta]^d$, Taylor's theorem guarantees the existence of a constant $c\in(0,1)$ such that
\begin{equation}
    D^\beta f(x) = \sum_{\abs{\gamma}\leq s-1-\abs{\beta}} \frac{D^\gamma D^\beta f(0)}{\gamma!} x^\gamma + \sum_{\abs{\gamma} = s-\abs{\beta}} \frac{D^\gamma D^\beta f(cx)}{\gamma!} x^\gamma. 
\end{equation}
The previous equalities, together with the multinomial theorem, then prove that
\begin{equation}
    \norm{D^\beta f-D^\beta \hat{f}}_{L^\infty([-\delta,\delta]^d)} \leq C_s \sum_{\abs{\gamma} = s-\abs{\beta}} \frac{\delta^{\abs{\gamma}} }{\gamma!} = \frac{C_s (d\delta)^{s-\abs{\beta}}}{(s-\abs{\beta})!},
\end{equation}
where $C_s := \abs{f}_{W^{s,\infty}([-\delta,\delta]^d)} $. Under the assumption that $\delta<1/d$ we then can conclude that
\begin{equation}
     \norm{f-\hat{f}}_{W^{k,\infty}([-\delta,\delta]^d)} \leq \frac{C_s (d\delta)^{s-k}}{(s-k)!}.
\end{equation}
\end{proof}

\bibliography{ref}

\end{document}